\documentclass[a4paper,twoside]{article}
\usepackage{a4}
\usepackage{amssymb}
\usepackage{amsmath}
\usepackage{upref}
\usepackage[active]{srcltx}
\allowdisplaybreaks[2] 
\usepackage[dvips,colorlinks,citecolor=blue,linkcolor=blue]{hyperref}
\newcount\minutes \newcount\hours
\hours=\time
\divide\hours 60
\minutes=\hours
\multiply\minutes -60
\advance\minutes \time
\newcommand{\klockan}{\the\hours:{\ifnum\minutes<10 0\fi}\the\minutes}
\newcommand{\tid}{\today\ \klockan}
\newcommand{\prtid}{\smash{\raise 10mm \hbox{\LaTeX ed \tid}}}
\renewcommand{\prtid}{}
\makeatletter
\def\sectionmark#1{} 
\def\subsectionmark#1{}
\newcommand{\sectnr}{\ifnum \c@secnumdepth >\z@
                 \thesection.\hskip 1em\relax \fi}
\def\@evenhead{\footnotesize\rm\thepage\hfil\leftmark\hfil\llap{\prtid}}
\def\@oddhead{\footnotesize\rm\rlap{\prtid}\hfil\rightmark\hfil\thepage}
\def\tableofcontents{\section*{Contents} 
 \@starttoc{toc}}
\makeatother
\makeatletter
\def\@biblabel#1{#1.}
\makeatother
\makeatletter
\let\Thebibliography=\thebibliography
\renewcommand{\thebibliography}[1]{\def\@mkboth##1##2{}\Thebibliography{#1}
\addcontentsline{toc}{section}{References}
\frenchspacing 
\setlength{\@topsep}{0pt}
\setlength{\itemsep}{0pt}%
\setlength{\parskip}{0pt plus 2pt}%
}
\makeatother
\makeatletter
\def\mdots@{\mathinner.\nonscript\!.%
 \ifx\next,.\else\ifx\next;.\else\ifx\next..\else
 \nonscript\!\mathinner.\fi\fi\fi}
\let\ldots\mdots@
\let\cdots\mdots@
\makeatother
\makeatletter
\let\Enumerate=\enumerate
\renewcommand{\enumerate}{\Enumerate%
\setlength{\@topsep}{0pt}
\setlength{\itemsep}{0pt}%
\setlength{\parskip}{0pt plus 1pt}%
\renewcommand{\theenumi}{\textup{(\alph{enumi})}}%
\renewcommand{\labelenumi}{\theenumi}%
}
\let\endEnumerate=\endenumerate
\renewcommand{\endenumerate}{\endEnumerate\unskip}
\makeatother
\makeatletter
\def\@seccntformat#1{\csname the#1\endcsname.\quad}
\makeatother
\newcommand{\authortitle}[2]{\author{#1}\title{#2}\markboth{#1}{#2}}
\newcommand{\art}[6]{{\sc #1, \rm #2, \it #3\/ \bf #4 \rm (#5), \mbox{#6}.}}
\newcommand{\artnopt}[6]{{\sc #1, \rm #2, \it #3\/ \bf #4 \rm (#5), \mbox{#6}}}
\newcommand{\auth}[2]{{#1, #2.}}
\newcommand{\artprep}[3]{{\sc #1, \rm #2, #3.}}
\newcommand{\artin}[3]{{\sc #1, \rm #2, in #3.}}
\newcommand{\book}[3]{{\sc #1, \it #2, \rm #3.}}
\newcommand{\private}[2]{\sc #1, \emph{Private communication}, #2.}
\newcommand{\AND}{{\rm and }}
\RequirePackage{amsthm}
\newtheoremstyle{descriptive}%
  {\topsep}   
  {\topsep}   
  {\rmfamily} 
  {}          
  {\bfseries} 
  {.}         
  { }         
  {}          
\newtheoremstyle{propositional}%
  {\topsep}   
  {\topsep}   
  {\itshape}  
  {}          
  {\bfseries} 
  {.}         
  { }         
  {}          
\newtheoremstyle{remarkstyle}%
  {\topsep}   
  {\topsep}   
  {\rmfamily}  
  {}          
  {\itshape} 
  {.}         
  { }         
  {}          
\theoremstyle{propositional}
\newtheorem{thm}{Theorem}[section]
\newtheorem{prop}[thm]{Proposition}
\newtheorem{lem}[thm]{Lemma}
\newtheorem{cor}[thm]{Corollary}
\theoremstyle{descriptive}
\newtheorem{deff}[thm]{Definition}
\newtheorem{example}[thm]{Example}
\newtheorem{remark}[thm]{Remark}
\makeatletter
\renewenvironment{proof}[1][\proofname]{\par
  \pushQED{\qed}%
  \normalfont 
  \trivlist
  \item[\hskip\labelsep
        \itshape
    #1\@addpunct{.}]\ignorespaces
}{%
  \popQED\endtrivlist\@endpefalse
}
\makeatother
\newcommand{\setm}{\setminus}
\renewcommand{\emptyset}{\varnothing}
\def\vint{\mathop{\mathchoice%
          {\setbox0\hbox{$\displaystyle\intop$}\kern 0.22\wd0%
           \vcenter{\hrule width 0.6\wd0}\kern -0.82\wd0}%
          {\setbox0\hbox{$\textstyle\intop$}\kern 0.2\wd0%
           \vcenter{\hrule width 0.6\wd0}\kern -0.8\wd0}%
          {\setbox0\hbox{$\scriptstyle\intop$}\kern 0.2\wd0%
           \vcenter{\hrule width 0.6\wd0}\kern -0.8\wd0}%
          {\setbox0\hbox{$\scriptscriptstyle\intop$}\kern 0.2\wd0%
           \vcenter{\hrule width 0.6\wd0}\kern -0.8\wd0}}%
          \mathopen{}\int}
\newcommand{\Cp}{{C_p}}
\DeclareMathOperator{\diam}{diam}
\DeclareMathOperator{\capp}{cap}
\newcommand{\cp}{\capp_p}
\newcommand{\cpmu}{\capp_{p,\mu}}
\DeclareMathOperator{\dist}{dist}
\DeclareMathOperator{\inter}{int}
\DeclareMathOperator{\Lip}{Lip}
\DeclareMathOperator{\lip}{lip}
\DeclareMathOperator{\interior}{int}
\DeclareMathOperator*{\essinf}{ess\,inf}
\newcommand{\bdry}{\partial}
\newcommand{\bdy}{\bdry}
\newcommand{\loc}{_{\rm loc}}
{\catcode`p =12 \catcode`t =12 \gdef\eeaa#1pt{#1}}      
\def\accentadjtext#1{\setbox0\hbox{$#1$}\kern   
                \expandafter\eeaa\the\fontdimen1\textfont1 \ht0 }
\def\accentadjscript#1{\setbox0\hbox{$#1$}\kern 
                \expandafter\eeaa\the\fontdimen1\scriptfont1 \ht0 }
\def\accentadjscriptscript#1{\setbox0\hbox{$#1$}\kern   
                \expandafter\eeaa\the\fontdimen1\scriptscriptfont1 \ht0 }
\def\accentadjtextback#1{\setbox0\hbox{$#1$}\kern       
                -\expandafter\eeaa\the\fontdimen1\textfont1 \ht0 }
\def\accentadjscriptback#1{\setbox0\hbox{$#1$}\kern     
                -\expandafter\eeaa\the\fontdimen1\scriptfont1 \ht0 }
\def\accentadjscriptscriptback#1{\setbox0\hbox{$#1$}\kern 
                -\expandafter\eeaa\the\fontdimen1\scriptscriptfont1 \ht0 }
\def\itoverline#1{{\mathsurround0pt\mathchoice
        {\rlap{$\accentadjtext{\displaystyle #1}
                \accentadjtext{\vrule height1.593pt}
                \overline{\phantom{\displaystyle #1}
                \accentadjtextback{\displaystyle #1}}$}{#1}}
        {\rlap{$\accentadjtext{\textstyle #1}
                \accentadjtext{\vrule height1.593pt}
                \overline{\phantom{\textstyle #1}
                \accentadjtextback{\textstyle #1}}$}{#1}}
        {\rlap{$\accentadjscript{\scriptstyle #1}
                \accentadjscript{\vrule height1.593pt}
                \overline{\phantom{\scriptstyle #1}
                \accentadjscriptback{\scriptstyle #1}}$}{#1}}
        {\rlap{$\accentadjscriptscript{\scriptscriptstyle #1}
                \accentadjscriptscript{\vrule height1.593pt}
                \overline{\phantom{\scriptscriptstyle #1}
                \accentadjscriptscriptback{\scriptscriptstyle #1}}$}{#1}}}}
\def\cprime{{\mathsurround0pt$'$}}
\newcommand{\ga}{\gamma}
\newcommand{\Ga}{\Gamma}
\newcommand{\dmu}{d\mu}
\newcommand{\de}{\delta}
\newcommand{\eps}{\varepsilon}
\newcommand{\la}{\lambda}
\newcommand{\Om}{\Omega}
\newcommand{\clAp}{{\itoverline{A}\mspace{1mu}}^p}
\newcommand{\clGp}{{\itoverline{G}\mspace{1mu}}^p}
\newcommand{\p}{{$p\mspace{1mu}$}}   
\newcommand{\R}{\mathbf{R}}
\newcommand{\eR}{{\overline{\R}}}
\newcommand{\Gt}{\widetilde{G}}
\newcommand{\limplus}{{\mathchoice{\raise.17ex\hbox{$\scriptstyle +$}}
                {\raise.17ex\hbox{$\scriptstyle +$}}
                {\raise.1ex\hbox{$\scriptscriptstyle +$}}
                {\scriptscriptstyle +}}}
\newcommand{\Np}{N^{1,p}}
\newcommand{\Nploc}{N^{1,p}\loc}
\newcommand{\Lploc}{L^{p}\loc}
\newcommand{\Ct}{\widetilde{C}}
\newcommand{\HP}{H^{1,p}}
\newcommand{\cpvar}{\widetilde{{\rm cap}}_p}
\makeatletter
\newcommand{\setcurrentlabel}[1]{\def\@currentlabel{#1}}
\makeatother
\newcounter{saveenumi}
\numberwithin{equation}{section}
\newenvironment{ack}{\medskip{\it Acknowledgement.}}{}

\begin{document}

\authortitle{Anders Bj\"orn and Jana Bj\"orn}
{The variational capacity with respect to nonopen sets in metric spaces}
\author{
Anders Bj\"orn \\
\it\small Department of Mathematics, Link\"opings universitet, \\
\it\small SE-581 83 Link\"oping, Sweden\/{\rm ;}
\it \small anders.bjorn@liu.se
\\
\\
Jana Bj\"orn \\
\it\small Department of Mathematics, Link\"opings universitet, \\
\it\small SE-581 83 Link\"oping, Sweden\/{\rm ;}
\it \small jana.bjorn@liu.se
}

\date{}
\maketitle

\noindent{\small
{\bf Abstract}. 
We pursue a systematic treatment of the variational capacity 
on metric spaces and give
full proofs of its basic properties. 
A novelty is that we study it with respect to nonopen sets,
which is important for 
Dirichlet and obstacle problems on nonopen sets, with applications
in fine potential theory.
Under standard assumptions on the underlying metric space, we show that 
the variational capacity is
a Choquet capacity and we provide several equivalent definitions for it.
On open sets in weighted $\R^n$ it is shown to coincide with the
usual variational capacity considered in the literature.

Since some desirable properties fail on general nonopen sets, we
introduce a related capacity which turns out to be a Choquet 
capacity in general metric spaces and for many sets coincides
with the variational capacity.
We provide examples demonstrating various properties of both capacities and 
counterexamples for when they fail.
Finally, we discuss how a change of the underlying metric space influences 
the variational capacity and its minimizing functions.
} 

\bigskip
\noindent
{\small \emph{Key words and phrases}: 
Choquet capacity,
doubling measure, 
metric space,  Newtonian space, nonlinear, outer capacity,
\p-harmonic, Poincar\'e inequality, potential theory,
quasicontinuous, Sobolev space,
upper gradient, variational capacity.
}


\section{Introduction}

The variational capacity $\cp(A,\Om)$ has been used extensively 
in nonlinear potential theory on $\R^n$, e.g.\ in the monographs
Heinonen--Kilpel\"ainen--Martio~\cite{HeKiMa}
and Mal\'y--Ziemer~\cite{MaZi}. 
Roughly speaking it is the energy of the \p-harmonic function in $\Om\setm A$
with zero boundary values on $\bdry\Om$ and boundary values~1 on $A$,
but its exact definition in $\R^n$ is usually done in three steps:  
through a minimization problem for compact $A$ and then by inner and outer
regularity for open and arbitrary sets $A$.
Also the choice of admissible functions in the minimization problem
varies in the literature.

The variational capacity is closely related to capacitary potentials 
and thus naturally appears in the Wiener criterion for boundary regularity
of \p-harmonic functions (Maz\cprime ya~\cite{Maz70}, 
Lindqvist--Martio~\cite{LiMa}, Kilpel\"ainen--Mal\'y~\cite{KilMa94}),
even though in unweighted $\R^n$ with $p<n$ it can be equivalently
replaced by the Sobolev capacity.
Through the Wiener integral it also plays an important role in the definition
of thinness and in nonlinear fine potential theory (\cite{HeKiMa}, \cite{MaZi}).

In nonlinear potential theory on metric spaces, 
the variational capacity $\cp$
(Definition~\ref{deff-varcap}) has been used in different contexts in e.g.\ 
\cite{ABclass}, \cite{BBnonopen}, \cite{BBMP}, \cite{BBS},
 \cite{BjIll}--\cite{BMS},
\cite{Fa}, \cite{KiMa03}, \cite{korte08}, \cite{KoMaSh}
and \cite{martioRep09}.
Very few of even the basic properties of $\cp$ have been given full proofs 
in the metric space literature, 
even though they must be known to experts in the field.

In this paper we pursue a systematic treatment of the variational capacity 
on metric spaces and give
full proofs of its basic properties. 
(Recently some of these results were for open $E$
included in the monograph Bj\"orn--Bj\"orn~\cite{BBbook}.)
A novelty here is that we consider $\cp(A,E)$ when $E$ is not open, which
does not seem to have been considered earlier.
This is motivated by the study of obstacle and Dirichlet problems
on nonopen sets in 
Bj\"orn--Bj\"orn~\cite{BBnonopen}, where some of the results
in this paper are used.
Another motivation is the use of $\cp$ in the development of fine potential
theory on metric spaces which has been touched upon in~\cite{BBnonopen}
and which we want to pursue in forthcoming papers.

Let $X$ be a metric space equipped with 
a Borel measure $\mu$.
After presenting some background results in Section~\ref{sect-prelim},
we define the variational capacity $\cp(A,E)$
in Section~\ref{sect-cp} and establish
its basic properties holding in full generality, such as
countable and strong subadditivity.
The definition of $\cp$ on metric spaces is more straightforward than 
in $\R^n$, due to the use of Newtonian spaces which are the 
natural (and in some aspects better) substitutes for Sobolev spaces in the 
setting of metric spaces.
This straightforward approach is essential when studying 
the variational capacity $\cp(A,E)$
with respect to nonopen $E$.
This on the other hand means that the outer regularity 
\begin{equation} \label{eq-outer-reg}
\cp(A,E)=\inf_{\substack{G \text{ open} \\  A\subset G \subset E}} \cp(G,E)
\end{equation}
on metric spaces is not a direct consequence of the definition
and can be proved only under additional assumptions.
In particular, in Theorem~\ref{thm-outercap-cp}
we show that $\cp$, $p>1$, is an outer capacity
for sets $A \subset \interior E$, provided that all Newtonian functions
are quasicontinuous (which holds e.g.\ under the standard assumptions
that $X$ is complete and the measure
$\mu$ is doubling and supports a $(1,p)$-Poincar\'e inequality).
As a consequence of this result we obtain the equality between the
capacity of a set and its fine closure.
Together with other properties, which we prove here, the outer regularity 
implies that $\cp(\,\cdot\,,E)$ is a Choquet capacity and all Borel sets
are capacitable, see Theorem~\ref{thm-Choq-cap}. 

This and the above outer regularity result are then used 
in Section~\ref{sect-eq-Rn} to show that
our variational capacity
coincides with the variational capacity on (weighted and unweighted) 
$\R^n$ considered in Heinonen--Kilpel\"ainen--Martio~\cite{HeKiMa}.
Since these capacities are defined in different ways 
and using different admissible functions, 
their equality is not straightforward but relies on the above mentioned
Theorems~\ref{thm-outercap-cp} and~\ref{thm-Choq-cap}. 
Another ingredient is that the Newtonian spaces on weighted $\R^n$
coincide with the usual
weighted Sobolev spaces considered in \cite{HeKiMa}, with equal norms.
The proof of this relies on a deep result of Cheeger~\cite{Cheeg}
and is given in Bj\"orn--Bj\"orn~\cite{BBbook}.

If $E$ is open and the metric space $X$ satisfies the above standard
assumptions, then Theorems~\ref{thm-outercap-cp} and~\ref{thm-Choq-cap} 
imply that $\cp$ on $X$ can be
defined without Newtonian spaces using only elementary properties
of Lipschitz functions in a similar way as in $\R^n$, see 
Section~\ref{sect-alt-def}.
This is however not possible for nonopen $E$ and not known in more
general spaces.
We also give examples showing that the outer regularity~\eqref{eq-outer-reg}
of $\cp$ is not true for arbitrary $A \subset E$.
This suggests the following definition of a related capacity:
\begin{equation}  \label{eq-def-cpvar}
   \cpvar(A,E)=\inf_{\substack{G \text{ relatively open} \\  A\subset G \subset E}} \cp(G,E),
\end{equation}
which is obviously outer for all sets $A \subset E$ and turns out to be
a Choquet capacity for many $A\subset E$ 
without any additional 
assumptions on $X$, see Theorem~\ref{thm-Choquet-cpvar}.
If $E$ is locally compact then this holds for all $A\subset E$.
Thus, $\cpvar$ seems to be a ``better'' modification of $\cp$, and
in Proposition~\ref{prop-cp=cpvar} we show that under certain
assumptions, $\cpvar(A,E)=\cp(A,E)$ or $\cpvar(A,E)=\infty$.
At the same time, 
some natural basic properties can fail for $\cpvar$, 
see Section~\ref{sect-alt-def}.
Moreover, in connection with e.g.\ Adams' criterion
in Bj\"orn--Bj\"orn~\cite{BBnonopen}, it is our capacity $\cp$ 
which is needed and its role cannot be played by any other nonequivalent 
capacity.

We finish the paper with a short discussion on how changes of 
the underlying space $X$ influence $\cp$, see Example~\ref{ex-vary-X}.

\begin{ack}
The  authors were supported by the Swedish Research Council
and belong to the
European Science
Foundation Networking Programme \emph{Harmonic and Complex Analysis and
Applications}
and to the Scandinavian Research Network \emph{Analysis and Application}.
\end{ack}

\section{Notation and preliminaries}
\label{sect-prelim}

We assume throughout the paper that $X=(X,d,\mu)$ is a 
metric space equipped
with a metric $d$ and a  measure $\mu$ such that
\[
     0 <  \mu(B)< \infty
\]
for all balls 
$B=B(x_0,r):=\{x\in X: d(x,x_0)<r\}$ in~$X$
(we make the convention that balls are nonempty and open).
We also assume that $1 \le p<\infty$ and that 
$E \subset X$ is a bounded  set.

The $\sigma$-algebra on which $\mu$ is defined
is obtained by completion of the Borel $\sigma$-algebra.
The measure $\mu$ is \emph{doubling} if there exists
a constant $C>0$ such that 
\begin{equation*}
        0 < \mu(2B) \le C \mu(B) < \infty,
\end{equation*}
where $\lambda B=B(x_0,\lambda r)$.  
Note that if $\mu$ is doubling then 
$X$ is complete if and only if $X$ is proper,
i.e.\ all bounded closed sets are compact.

A \emph{curve} is a continuous mapping from an interval.
We will only consider curves which are nonconstant, compact and rectifiable.
A curve can thus be parameterized by its arc length $ds$. 
We follow Heinonen and Koskela~\cite{HeKo98} in introducing
upper gradients as follows (they called them
very weak gradients).

\begin{deff} \label{deff-ug}
A nonnegative Borel function $g$ on $X$ is an \emph{upper gradient} 
of an extended real-valued function $f$
on $X$ if for all (nonconstant, compact and rectifiable) curves  
$\gamma: [0,l_{\gamma}] \to X$,
\begin{equation} \label{ug-cond}
        |f(\gamma(0)) - f(\gamma(l_{\gamma}))| \le \int_{\gamma} g\,ds,
\end{equation}
where we make the convention that the left-hand side is $\infty$ 
whenever both terms therein are infinite.
If $g$ is a nonnegative measurable function on $X$
and if (\ref{ug-cond}) holds for \p-almost every curve (see below), 
then $g$ is a \emph{\p-weak upper gradient} of~$f$. 
\end{deff}

Here and in what follows, we say that a property holds for 
\emph{\p-almost every curve}
if it fails only for a curve family $\Ga$ with zero \p-modulus, 
i.e.\ there exists $0\le\rho\in L^p(X)$ such that 
$\int_\ga \rho\,ds=\infty$ for every curve $\ga\in\Ga$.
Note that a \p-weak upper gradient \emph{need not\/} be a Borel function,
only measurable.
It is implicitly assumed that $\int_{\gamma} g\,ds$ is
defined (with a value in $[0,\infty]$) for
\p-almost every  curve $\ga$,
although this is in fact a consequence of the measurability.
For proofs of these and all other facts 
in this section 
we refer to Bj\"orn--Bj\"orn~\cite{BBbook}. 
(Some of the references we mention below may not 
provide a proof in the generality considered here, but 
such proofs are given in \cite{BBbook}.)

The \p-weak upper gradients were introduced in
Koskela--MacManus~\cite{KoMc}.
They also showed that if $g \in \Lploc(X)$ is a \p-weak upper gradient of $f$,
then one can find a sequence $\{g_j\}_{j=1}^\infty$
of upper gradients of $f$ such that $g_j-g \to 0$ in $L^p(X)$.
If $f$ has an upper gradient in $\Lploc(X)$, then
it has a \emph{minimal \p-weak upper gradient} $g_f \in \Lploc(X)$
in the sense
 that for every \p-weak upper gradient $g \in \Lploc(X)$ of $f$ we have
$g_f \le g$ a.e.,
see Shan\-mu\-ga\-lin\-gam~\cite{Sh-harm}
and Haj\l asz~\cite{Haj03}.
The minimal \p-weak upper gradient is well defined
up to an equivalence class in the cone of nonnegative functions in $\Lploc(X)$.
Following Shanmugalingam~\cite{Sh-rev}, 
we define a version of Sobolev spaces on the metric space $X$.

\begin{deff}
The \emph{Newtonian space} on $X$ is 
\[
        \Np (X) = \{u: \|u\|_{\Np(X)} <\infty \},
\]
where
\[
        \|u\|_{\Np(X)} = \biggl( \int_X |u|^p \, \dmu 
                + \int_X g_u^p \, \dmu \biggr)^{1/p}
\]
for an everywhere defined measurable function $u:X\to\overline{\R}$ 
having an upper gradient in $\Lploc(X)$.
\end{deff}

The space $\Np(X)/{\sim}$, where  $u \sim v$ if and only if $\|u-v\|_{\Np(X)}=0$,
is a Banach space
and a lattice, see Shan\-mu\-ga\-lin\-gam~\cite{Sh-rev}.
For a measurable set $E \subset X$, the space $\Np(E)$ 
is defined
by considering $E$ as a metric space on its own.
Let us here point out that we assume that
functions in Newtonian spaces are defined everywhere,
and not just up to equivalence classes in $L^p$.

If $u, v \in \Nploc(X)$, then their minimal \p-weak upper gradients coincide
a.e.\ in the set $\{x \in X : u(x)=v(x)\}$,
in particular $g_{\min\{u,c\}}=g_u \chi_{\{u < c\}}$ a.e.\  for $c \in \R$.
Moreover, $g_{uv} \le |u|g_v + |v|g_u$. 

\begin{deff}
The
\emph{Sobolev capacity} of a set $A \subset X$ 
is the number 
\begin{equation*} 
  \Cp (A) =\inf    \|u\|_{\Np(X)}^p,
\end{equation*}
where the infimum is taken over all $u\in \Np (X) $ such that
$u=1$ on $A$.

We say that a property 
holds \emph{quasieverywhere} (q.e.)\ 
if the set of points  for which it fails
has capacity zero.
\end{deff}

The Sobolev capacity was introduced and used for Newtonian spaces
in Shanmugalingam~\cite{Sh-rev}.
It is countably subadditive and the correct gauge 
for distinguishing between two Newtonian functions. 
If $u \in \Np(X)$ and $v:X \to \eR$,
then $u \sim v$ if and only if $u=v$ q.e.
Moreover, 
if $u,v \in \Np(X)$ and
$u= v$ a.e., then $u=v$ q.e.
The proofs of properties for $\Cp$ are similar or easier than the proofs
of the corresponding properties for the variational capacity $\cp$ 
presented in this paper.
Note also that if $\Cp(E)=0$, then \p-almost every curve in $X$ avoids $E$,
by e.g.\ Lemma~3.6 in Shanmugalingam~\cite{Sh-rev}
or Proposition~1.48 in Bj\"orn--Bj\"orn~\cite{BBbook}.

To be able to compare the boundary values of Newtonian functions
we need a Newtonian space with zero boundary values.
We let
\[
\Np_0(E)=\{f|_{E} : f \in \Np(X) \text{ and }
        f=0 \text{ on } X \setm E\}.
\]
One can replace the assumption ``$f=0$ on $X \setm E$''
with ``$f=0$ q.e.\ on $X \setm E$''
without changing the obtained space $\Np_0(E)$.
Functions from $\Np_0(E)$ can be extended by zero q.e.\ in $X\setm E$ and we
will regard them in that sense if needed.
Note that if $\Cp(X \setminus E) = 0$, then $\Np_0(E) = \Np(E) = \Np(X)$,
since \p-almost every curve in $X$ avoids $X\setm E$.

The following Poincar\'e inequality is often assumed in the literature.
Because of the dilation $\la$ in the right-hand side, it is sometimes called
weak Poincar\'e inequality.

\begin{deff} \label{def-PI}
We say that $X$ supports a \emph{$(q,p)$-Poincar\'e inequality},
$q \ge 1$, if
there exist constants $C>0$ and $\lambda \ge 1$
such that for all balls $B \subset X$
and all integrable $u\in\Np(X)$,
\begin{equation} \label{PI-ineq}
        \biggl(\vint_{B} |u-u_B|^q \,\dmu\biggl)^{1/q}
        \le C (\diam B) \biggl( \vint_{\lambda B} g_u^{p} \,\dmu \biggr)^{1/p},
\end{equation}
where $ u_B 
 :=\vint_B u \,\dmu 
:= \int_B u\, d\mu/\mu(B)$.
\end{deff}

Using the above-mentioned results on \p-weak upper gradients from
Koskela--MacManus~\cite{KoMc},
it is easy to see that \eqref{PI-ineq} can equivalently be required for 
all upper gradients $g$ of $u$.
If $X$ supports a $(1,p)$-Poincar\'e inequality and $\mu$ is doubling,
then by Theorem~5.1 in Haj\l asz--Koskela~\cite{HaKo},
it supports a $(q,p)$-Poincar\'e inequality for some $q>p$,
and in particular a $(p,p)$-Poincar\'e inequality.
If $X$ is moreover complete 
then Lipschitz functions
are dense in $\Np(X)$, see Shan\-mu\-ga\-lin\-gam~\cite{Sh-rev}, 
and functions in $\Np(X)$
as well as in $\Np(\Om)$ are quasicontinuous 
(see Theorem~\ref{thm-quasicont} below).
It also follows
that $\Np_0(\Om)$ for open $\Om$
can equivalently be defined as the closure of
Lipschitz functions with compact support in $\Om$,
see Shanmugalingam~\cite{Sh-harm} or Theorem~5.45 in 
Bj\"orn--Bj\"orn~\cite{BBbook}.
For a general set $E$ this is not always possible
and our definition of $\Np_0(E)$ seems to be the natural one.

Moreover, if $X$ is unweighted $\R^n$ and $u \in \Np(X)$, then
$g_u=|\nabla u|$ a.e., where $\nabla u$ is the distributional gradient
of $u$.
This means that in the Euclidean setting, $\Np(\Om)$ for open $\Om \subset \R^n$
 is the 
refined Sobolev space as defined on p.\ 96 of
Heinonen--Kilpel\"ainen--Martio~\cite{HeKiMa}.
See Haj\l asz~\cite{Haj03} or 
Appendix~A.1 in \cite{BBbook} for a full
proof of this fact for unweighted $\R^n$, and
Appendix~A.2 in \cite{BBbook} for a proof for weighted $\R^n$
(requiring $p>1$).
See also Theorem~\ref{thm-equiv-HeKiMa} 
and its proof for further details.

A function $u :  X \to \eR$ is \emph{quasicontinuous}
if for every $\eps>0$ there is an \emph{open} set $G$
with $\Cp(G)<\eps$ such that $u|_{X \setm G}$ is real-valued continuous.

\begin{thm} \label{thm-quasicont}
\textup{(Bj\"orn--Bj\"orn--Shan\-mu\-ga\-lin\-gam~\cite{BBS5})}
Let $X$ be proper,  $\Om \subset X$ be open, and assume that 
continuous functions are dense in $\Np(X)$
\textup{(}which in particular holds if 
$X$ is complete and supports a\/ $(1,p)$-Poincar\'e inequality,
and  $\mu$ is doubling\/\textup{)}.
Then every $u\in\Np(\Om)$ is quasicontinuous in\/ $\Om$.
\end{thm}

In several of our results the main assumption needed in the proof
is that all functions in $\Np(X)$ are quasicontinuous. 
The theorem above is the main result guaranteeing this. 
See 
Bj\"orn--Bj\"orn~\cite{BBbook}, Section~5.1, for several
examples, not supporting Poincar\'e inequalities, when this holds.
Moreover, in the other extreme situation when there are no curves in $X$,
then $\Np(X)=L^p(X)$ and thus the quasicontinuity follows directly
from Luzin's theorem. 
(Incidentally, Luzin's theorem (on $\R$) was first
obtained by Vitali~\cite{vitali} in 1905, while
Luzin~\cite{luzin} obtained it in 1912.)
In fact, there is no example of a nonquasicontinuous
Newtonian function, see Open problems~5.34 and~5.35 in \cite{BBbook}.

\section{Definition of  \texorpdfstring{$\cp$}{} and basic properties}
\label{sect-cp}

Recall that we assume that $E \subset X$ is a  bounded set.

\begin{deff} \label{deff-varcap}
For $A \subset E$ we define the variational capacity
\begin{equation*} 
  \cp (A,E) =\inf   \int_X g_u^p \, d\mu,
\end{equation*}
where the infimum is taken over all $u\in \Np_0 (E) $ such that
$u \ge 1$ on $A$.

(Here and later we use the usual convention that $\inf \emptyset = \infty$.)
\end{deff}

The infimum can equivalently be taken over 
all nonnegative $u\in \Np_0 (E) $ such that
$ u = 1$ on $A$.
If $E$ is measurable we may also equivalently integrate over $E$ instead
of $X$. 

Note that as $\Np_0(E)\subset\Np(X)$, it is natural to consider the minimal
\p-weak upper gradient $g_u$ with respect to $X$.
On the other hand, by Proposition~3.10 
in Bj\"orn--Bj\"orn~\cite{BBnonopen},
$g_u$ is also minimal as a \p-weak upper gradient on $E$ 
(if $E$ is measurable).

The variational capacity $\cp (A,E)$ has been used and studied earlier for 
bounded open $E$ in metric spaces by e.g.\ 
Bj\"orn--MacManus--Shanmugalingam~\cite{BMS}
and J.~Bj\"orn~\cite{BjIll},~\cite{JB-pfine}.
It can also be regarded as the condenser capacity 
$\cp (X\setm E,A,X)$, in which the test functions satisfy 
$u=0$ in $X\setm E$ and $u=1$ on $A$.
Such a capacity has been studied on metric spaces by
Heinonen--Koskela~\cite{HeKo98},
Kallunki--Shan\-mu\-ga\-lin\-gam~\cite{KaSh}
and 
Adamowicz--Bj\"orn--Bj\"orn--Shan\-mu\-ga\-lin\-gam~\cite{ABBSprime}.

A novelty here is that we consider nonopen $E$.
However most of the results below have not been given full proofs
in the Newtonian literature even for open $E$.
The following result 
shows that, under quite general assumptions, the
zero sets of $\Cp$ and $\cp$ are the same.
It generalizes Lemma~3.3 in \cite{BjIll}.

\begin{lem} \label{lem-Cp<=>cp}
Assume that $X$ supports a\/  $(p,p)$-Poincar\'e inequality for $\Np_0$,
see below,
and that $\Cp(X \setm E)>0$.
Let $A \subset E$. Then
$\Cp(A)=0$ if and only if $\cp(A,E)=0$.
\end{lem}

Observe that if $\Cp(X \setm E)=0$ (and thus $X$ is bounded),
 then $1 \in \Np_0(E)$,
making $\cp(A,E)=0$ for all sets $A \subset E$.

The proof shows that the necessity  holds 
without any assumptions on $X$.
For the sufficiency we need a Poincar\'e inequality, but it is enough with
a considerably weaker Poincar\'e inequality than the one in 
Definition~\ref{def-PI}.
Note also that doubling is not needed.

\begin{deff} \label{def-relax-PI}
We say that $X$ supports a \emph{$(p,p)$-Poincar\'e inequality for $\Np_0$}
if for every bounded $E\subset X$ with $\Cp(X\setm E)>0$
there exists $C_E>0$  such that for all  $u\in\Np_0(E)$
\textup{(}extended by $0$ outside $E$\textup{)},
\begin{equation}
        \int_X |u|^p \,\dmu  
                \le C_E \int_{X} g_u^p \,\dmu.
\label{PI-NP0}
\end{equation}
\end{deff}

A direct consequence is that $\|u\|_{\Np(X)}^p \le \Ct_E \|g_u\|^p_{L^p(X)}$ 
for $u\in\Np_0(E)$.

See Bj\"orn--Bj\"orn~\cite{BBnonopen}
for  further discussion of this Poincar\'e inequality,
in particular a proof that it follows from the 
$(p,p)$-Poincar\'e inequality.

\begin{proof}  [Proof of Lemma~\ref{lem-Cp<=>cp}.]
Assume first that $\Cp(A)=0$.
Then $\chi_A \in\Np(X)$ and consequently $\chi_A \in\Np_0(E)$.
Since $g_{\chi_A}=0$ a.e., it follows that $\cp(A,E)=0$.

Conversely, assume that $\cp(A,E)=0$
and let $\eps>0$.
Then there is $u \in \Np_0(E)$ such that $u=1$ on $A$
and $\int_{X} g_u^p \,\dmu< \eps$.
The $(p,p)$-Poincar\'e inequality for $\Np_0$ then yields that
\[
   \Cp(A) \le \|u\|_{\Np(X)}^p
	\le \Ct_E \int_X g_u^p\,d\mu
	< \Ct_E \eps.
\]
Letting $\eps \to 0$ concludes the proof.
\end{proof}

Let us collect the main general properties of the capacity $\cp$.
Observe that these 
properties all hold in full generality 
(apart from the requirement $p>1$ in \ref{cp-Choq-E-sum}).

\begin{thm} \label{thm-cp}
Assume that
$A_1, A_2,\ldots \subset E$.
Then the following properties hold\/\textup{:}
\begin{enumerate}
\renewcommand{\theenumi}{\textup{(\roman{enumi})}}%
\item \label{cp-emptyset-sum}
  $\cp(\emptyset,E)=0$\textup{;}
\item \label{cp-subset-sum}
  if $A_1 \subset A_2 \subset E$\textup{,}
 then\/ $\cp(A_1,E) \le \cp(A_2,E)$\textup{;}
\item \label{cp-subset-sum-2}
  if $A \subset E_1 \subset E_2$\textup{,}
 then\/ $\cp(A,E_2) \le \cp(A,E_1)$\textup{;}
\item \label{cp-strong-subadd}
$\cp$ is strongly subadditive, i.e.
\[ 
   \cp(A_1 \cup A_2,E) + \cp(A_1 \cap A_2,E) \le \cp(A_1,E)+\cp(A_2,E);
\]
\item \label{cp-subadd-sum}
$\cp$ is countably subadditive\/
\textup{(}and is also an 
outer measure\/\textup{)}, i.e.
  \[
      \cp\biggl(\bigcup_{i=1}^\infty A_i,E\biggr) 
          \le \sum_{i=1}^\infty \cp(A_i,E)
      \textup{;}
  \]
\item \label{cp-Choq-E-sum}
if\/ $1<p<\infty$ and  $A_1 \subset A_2 \subset \cdots \subset E$\textup{,}  then
\[
      \cp\biggl(\bigcup_{i=1}^\infty A_i,E\biggr) 
          = \lim_{i \to \infty} \cp(A_i,E)
      \textup{;}
  \]
\item \label{cp-F-bdyF}
\setcounter{saveenumi}{\value{enumi}}
if $F \subset E$ is closed\/ {\rm(}as a subset of $X$\/{\rm)}\textup{,} 
then\/ $\cp(F,E)=\cp(\bdy F,E)$.
\end{enumerate}
\end{thm}

Even if $E$ is open, \ref{cp-Choq-E-sum} is not true
(in general) for $p=1$. We
refer the reader to 
Bj\"orn--Bj\"orn~\cite{BBbook} for a counterexample  due to 
Korte~\cite{korte-private}
(it also applies to $\cpvar$ from~\eqref{eq-def-cpvar}).

To prove \ref{cp-subadd-sum} we need the following simple 
lemma, which probably belongs to folklore.
It is a special case of Lemma~1.52 in \cite{BBbook},
but can be proved more easily along the lines of 
the proof of Lemma~1.28 in \cite{BBbook}.

\begin{lem} \label{lem-sup-u_i-weak}
Let 
$u_i \le M\in\R$, $i=1,2,\ldots$, 
be functions with \p-weak upper gradients $g_i$.
Let further $u=\sup_i u_i$ and $g=\sup_i g_i$.
Then $g$ is a \p-weak upper gradient of $u$.
\end{lem}

\begin{proof}[Proof of Theorem~\ref{thm-cp}]
\ref{cp-emptyset-sum}--\ref{cp-subset-sum-2}
These statements are trivial.

\ref{cp-strong-subadd}
We may assume that the right-hand side is finite.
Let $\eps >0$.
We can thus find $u_j \in \Np_0(E)$, $\chi_{A_j} \le u_j \le 1$, such that
$\|g_{u_j}\|_{L^p(X)}^p < \cp(A_j,E)+\eps$, $j=1,2$.
Let $v=\max\{u_1,u_2\}$ and $w=\min\{u_1,u_2\}$.
Then $g_v=g_{u_1} \chi_{\{u_1 > u_2\}} + g_{u_2} \chi_{\{u_2 \ge u_1\}}$
and $g_w=g_{u_2} \chi_{\{u_1 > u_2\}} + g_{u_1} \chi_{\{u_2 \ge u_1\}}$.
Since $v$ and $w$ are admissible in the definition of the variational
capacity of $A_1\cup A_2$ and $A_1\cap A_2$, respectively, we obtain
\begin{align*}
   \cp(A_1 \cup A_2,E) + \cp(A_1 \cap A_2,E) 
     &  \le \int_X (g_v^p + g_w^p) \, d\mu 
       = \int_X (  g_{u_1}^p  + g_{u_2}^p) \, d\mu\\
      &< \cp(A_1,E)+\cp(A_2,E) + 2\eps.
\end{align*}
Letting $\eps \to 0$ completes the proof of 
\ref{cp-strong-subadd}.

\ref{cp-subadd-sum} 
We may assume that the right-hand side is finite.
Let $\eps>0$.
Choose $u_i \in \Np_0(E)$ with $\chi_{A_i} \le u_i \le 1$
such that
\[
    \|g_{u_i}\|_{L^p(X)}^p \le \cp(A_i,E)+ \frac{\eps}{2^i}.
\]
Let $u=\sup_i u_i$ and $g=\sup_i g_{u_i}$.
By Lemma~\ref{lem-sup-u_i-weak}, $g$ is a \p-weak upper gradient of $u$.
Clearly $u \ge 1$ on $\bigcup_{i=1}^\infty A_i$.
Hence
\begin{align*}
      \cp\biggl(\bigcup_{i=1}^\infty A_i,E\biggr) 
          & 
        \le \int_X \Bigl(\sup_i g_{u_i}\Bigr)^p \,d\mu 
       \le  \int_X \sum_{i=1}^\infty g_{u_i}^p \,d\mu 
      = \sum_{i=1}^\infty   \int_X g_{u_i}^p \,d\mu  
\\          & 
\le \sum_{i=1}^\infty \Bigl( \cp(A_i,E)+ \frac{\eps}{2^i}\Bigr) 
       =  \eps + \sum_{i=1}^\infty \cp(A_i,E).
\end{align*}
Letting $\eps \to 0$ completes the proof of \ref{cp-subadd-sum}.

\ref{cp-Choq-E-sum} 
Let $A=\bigcup_{i=1}^\infty A_i$.
That $\lim_{i \to \infty} \cp(A_i,E) \le \cp (A,E)$ follows from monotonicity,
and monotonicity also shows that the limit always exists.
Conversely, assume that $\lim_{i \to \infty} \cp(A_i,E) < \infty$.

We can 
find $u_i \in \Np_0(E)$ with $\chi_{A_i} \le u_i \le 1$ and  such that
\[
           \|g_{u_i}\|_{L^p(X)}^p  < \cp(A_i,E)+1/i.
\]
By 
Lemma~3.2 in Bj\"orn--Bj\"orn--Parviainen~\cite{BBP}
(or Lemma~6.2 in Bj\"orn--Bj\"orn~\cite{BBbook}), 
there are $u,g \in L^p(X)$, 
finite convex combinations
$v_j=\sum_{i=j}^\infty a_{j,i} u_i$
and a strictly increasing sequence of  indices $\{i_k\}_{k=1}^\infty$
such that 
both $u_{i_k} \to u$ and $g_{u_{i_k}} \to g$ weakly in $L^p(X)$, as $k \to \infty$,
$v_j \to u$ q.e., as $j\to \infty$,
and $g$ is a \p-weak upper gradient of $u$.
Without loss of generality $u=0$ outside of $E$.

It is clear that $v_j \ge \chi_{A_j}$
and thus $u \ge \chi_A$ q.e.
Let $v:=\max\{u,\chi_A\}=u$ q.e.
Then $g$ is a \p-weak upper gradient also of $v$.
As $v\ge\chi_A$, we obtain
\begin{align*}
   \cp(A,E) &\le  \|g\|_{L^p(X)}^p 
          \le \liminf_{k \to \infty}  \|g_{u_{i_k}}\|_{L^p(X)}^p \\
          &
           \le  \lim_{k \to \infty} (\cp(A_{i_k},E) + 1/i_k)
           =   \lim_{i \to \infty} \cp(A_i,E).
\end{align*}

\ref{cp-F-bdyF}
Let $u$ be admissible in the definition of 
$\cp(\bdy F,E)$. 
Without loss of generality we can assume that $0 \le u \le 1$ and that
$u=0$ in $X\setm E$.
Let 
\[
       v = \begin{cases}
           1& \text{in } F, \\
           u& \text{in } X \setm F.
           \end{cases}
\]
Then $\|v\|_{L^p(X)} \le \|u\|_{L^p(X)}+ \mu(F) < \infty$.
Let $\ga:[0,l_\ga] \to X$ be a curve such that \eqref{ug-cond} holds
for $u$ and $g_u$ on $\ga$ and all its subcurves. 
If $\ga\subset F$ or $\ga\subset X\setm F$, then it is straightforward
that \eqref{ug-cond} holds for $v$ and $g_u$ on $\ga$.
If $\ga$ intersects both $F$ and $X\setm F$, we can, by splitting
$\ga$ into parts if necessary, and possibly 
reversing the direction,
assume that $x=\ga(0)\in F$ and
$y=\ga(l_\ga)\in X\setm F$.
Letting $t=\sup\{\tau: \ga(\tau)\in F\}$ we have
that $\ga(t)\in \bdy F$ and
hence
\[
|v(x)-v(y)| = |u(\ga(t))-u(y)| \le \int_{\ga|_{[t,l_\ga]}} g_u\,ds
     \le \int_{\ga} g_u\,ds,
\]
i.e.\ \eqref{ug-cond} holds for $v$ and $g_u$ on $\ga$ as well.
Thus $g_u$ is a \p-weak upper gradient of $v$ and hence
$v\in\Np(X)$.

As $v=u=0$ in $X\setm E$, we have that
$v \in \Np_0(E)$ and 
\[
    \cp(F,E) \le \int_{X} g^p_v\, d\mu
    \le \int_{X} g^p_u\, d\mu.
\]
Taking infimum over all $u$, we see that
$\cp(F,E) \le \cp(\bdy F,E)$.
The converse inequality is trivial.
\end{proof}

\section{Outer and Choquet capacity}
\label{sect-outer-cp}

In $\R^n$, the capacity is usually defined in a way which automatically
makes it an outer capacity.
Our definition using Newtonian functions is more direct, but $\cp$ is
an outer capacity only under some additional assumptions on the 
underlying space.
The following theorem shows this.

\begin{thm} \label{thm-outercap-cp}
Assume that all functions in $\Np(X)$ are quasicontinuous\/
\textup{(}which in particular holds if 
$X$ is complete and supports a\/ $(1,p)$-Poincar\'e inequality,
and  $\mu$ is doubling\/\textup{)}.
If $p>1$, 
then\/ $\cp$ is an \emph{outer capacity} for sets in\/ $\interior E$, 
i.e.\ for every 
$A\subset \interior E$,
\begin{equation} \label{eq-outercap-cp}
\cp(A,E)=\inf_{\substack{G \text{ open} \\  A\subset G \subset E}} \cp(G,E).
\end{equation}
If $p=1$, then \eqref{eq-outercap-cp}
holds for all $A \subset E$ with\/ $\dist(A,X \setm E)>0$. 
\end{thm}

When $p=1$ we do not know if \eqref{eq-outercap-cp} holds for
arbitrary $A \subset \interior E$. 

In Bj\"orn--Bj\"orn--Shan\-mu\-ga\-lin\-gam~\cite{BBS5}, p.\ 1199,
\eqref{eq-outercap-cp} was observed, under the more restrictive
assumptions that $X$ is proper, $E$ is open 
and $A \Subset E$ (which is however not enough to obtain
Theorem~\ref{thm-equiv-HeKiMa}). 
In Theorem~\ref{thm-outercap-cp}, the requirement $A \Subset E$ has been
weakened, and for open $E$ the result now holds without any additional 
assumptions on $A$.
The proof is a fair bit more
involved in this case and uses e.g.\ the strong subadditivity of the capacity
and Theorem~\ref{thm-cp}\,\ref{cp-Choq-E-sum}
(it is here that $p>1$ is needed).

Note that it is possible to have $\cp(A,E)<\infty$ even if $A$ ``reaches'' 
to the boundary $\bdry E$ and both $A$ and $E$ are open, see 
Example~\ref{ex-cap-A-to-bdryE}.
Note also that there are no known examples of Newtonian functions
which are not quasicontinuous.

For general $A \subset E$ (i.e.\ such that $A \not\subset \interior E$),
\eqref{eq-outercap-cp} is impossible (unless $\cp(A,E)=\infty$)
 as  there are no open sets $G \supset A$ such that $\cp(G,E)$ is defined.
In this case the natural question would be if 
\begin{equation} \label{eq-outercap-cp-rel}
\cp(A,E)=\inf_{\substack{G \text{ relatively open} \\  A\subset G \subset E}} \cp(G,E).
\end{equation}

This is not true in general, see Example~\ref{ex-arc},
but can be true also when $E$ is nonopen, see Example~\ref{ex-dense}
and Proposition~\ref{prop-cp=cpvar} below.
Of course for open $E$ it follows from Theorem~\ref{thm-outercap-cp} 
(provided that $p>1$ and all functions in $\Np(X)$ are
quasicontinuous).

\begin{example}   \label{ex-cap-A-to-bdryE}
Let $E=(-1,1)\times(0,1) \subset X=\R^2$ (unweighted)
with $1 < p < 2$, and let 
$A=\{(x,y)\in E: |x|<y<\tfrac12\}$. 
Then the function $u(x,y)=\min\{y/|x|,1\}$ multiplied by the cut-off
function
$\eta(x,y)=\min\{2-4\max\{|x|,y\},1\}_\limplus$ belongs to
$\Np_0(E)$ and is admissible in the definition of $\cp(A,E)$.
\end{example}

\begin{example} \label{ex-arc}
Let $E=[-1,1]\times[0,1] \subset X=\R^2$ (unweighted)
with $1 < p \le 2$, and let $A=\{(0,0)\}$. 
Then any relatively open set $G \supset A$ contains an open subinterval
of the real axis.
Since all functions in $\Np(X)$ are absolutely continuous on \p-almost
every curve (by Shan\-mu\-ga\-lin\-gam~\cite{Sh-rev}),
it follows that 
there is no function $u \in \Np_0(E)$ such that $u=1 $ on G.
Thus the right-hand side in 
\eqref{eq-outercap-cp-rel} is infinite.
The left-hand side is however $0$, by Lemma~\ref{lem-Cp<=>cp}.
\end{example}

\begin{example} \label{ex-dense}
Let $E=B(0,1)\setm D  \subset X=\R^2$ (unweighted)
with $1 < p \le 2$, where $D$ is a countable dense subset of $B(0,1)$.
As $\interior E = \emptyset$, Theorem~\ref{thm-outercap-cp} is directly
applicable only for $A=\emptyset$.
At the same time $\Cp(D)=0$, which shows that $\Np_0(E)=\Np_0(B(0,1))$,
and thus that $\cp(A \cap E,E)=\cp(A,B(0,1))$ for $A \subset B(0,1)$.
Hence, Theorem~\ref{thm-outercap-cp} applied to $\cp(\,\cdot\,,B(0,1))$
shows that \eqref{eq-outercap-cp-rel} is in fact true in this case.
\end{example}

It may be worth pointing out that $\Np_0(E)$ is closely related
to the fine interior of $E$, see 
Bj\"orn--Bj\"orn~\cite{BBnonopen}.
In fact, $\Np_0(E)=\Np_0(B(0,1))$ in Example~\ref{ex-dense} holds
because the set $E$ therein is finely open.

To prove Theorem~\ref{thm-outercap-cp}, we shall need the following simple
lemma which is based on the strong subadditivity of the capacity.

\begin{lem}  \label {lem-Aj-Gj-eps}
Let $A_1\subset A_2\subset \cdots \subset E$ be arbitrary and such that\/
$\cp(A_j,E)<\infty$ for all $j=1,2,\ldots$\,.
For each  $j$ let moreover $G_j\subset E$ be such that
$G_j\supset A_j$ and
\[
\cp(G_j,E) \le \cp(A_j,E) + \eps_j,
\]
where $\eps_j>0$ are arbitrary.
Let $\Gt_k = \bigcup_{j=1}^k G_j$.
Then for all $k=1,2,\ldots$,
\[
\cp(\Gt_k,E) \le \cp(A_k,E) + \sum_{j=1}^k \eps_j.
\]
\end{lem}

\begin{proof}  
The lemma is clearly true for $k=1$. 
Assume that it holds for some $k\ge1$.
We then  have by the strong
subadditivity of $\cp$, see Theorem~\ref{thm-cp}\,\ref{cp-strong-subadd},
that
\begin{align*}
\cp(\Gt_{k+1},E) &= \cp(\Gt_k\cup G_{k+1},E) \\
    &\le \cp(\Gt_k,E) + \cp(G_{k+1},E) - \cp(\Gt_k\cap G_{k+1},E).
\end{align*}
Since $\Gt_k\cap G_{k+1} \supset A_k$, this together with the induction
assumption yields
\begin{align*}
\cp(\Gt_{k+1},E) 
   &\le \cp(A_k,E) + \sum_{j=1}^k \eps_j + \cp(A_{k+1},E) +\eps_{k+1} 
             - \cp(A_k,E) \\
   &= \cp(A_{k+1},E) + \sum_{j=1}^{k+1} \eps_j.\qedhere
\end{align*}
\end{proof}

\begin{proof} [Proof of Theorem~\ref{thm-outercap-cp}.]
That 
\[
    \cp(A,E)\le \inf_
     {\substack{G \text{ open} \\  A\subset G \subset E}} \cp(G,E)
\]
follows from the monotonicity of the capacity. 
The converse inequality is trivial if $\cp(A,E)=\infty$.
Assume therefore that $\cp(A,E)<\infty$.

Assume first that $\dist(A,X\setm E)>0$ and let
$0<d<\frac{1}{2}\dist(A, X \setm E)$.
Let $0<\eps<1$ and find $u\in\Np_0(E)$ such that $u \ge \chi_A$ and
\[
     \|g_u\|_{L^p(X)}^p < \cp(A,E)+\eps.
\]
As $u$ (extended by zero outside $E$)
is quasicontinuous in $X$,
there is an open set $V$ with $\Cp(V)^{1/p} < \eps$ such that
$u|_{X \setm V}$ is continuous.
Thus, there is an open set $U$ such that
\[
          U \setm V = \{x : u(x) > 1-\eps\} \setm V \supset A \setm V.
\]
We can also find $v \ge \chi_V$ 
with $\|v\|_{\Np(X)} < \eps$.
Let $\eta(x)=\min\{1,2-\dist(x,A)/d\}_\limplus$.
Note that 
$\eta\in\Np_0(E)$, $0\le\eta\le1$, $g_\eta\le 1/d$
and $\eta=1$ in the open neighbourhood
$W:=\{x\in E:  \dist(x,A)<d\}$ of  $A$.
Then 
\[
\|g_{\eta v}\|_{L^p(X)} 
       \le \|g_{v}\|_{L^p(X)} + \frac1d \|v\|_{L^p(X)}
\le \biggl( 1+ \frac1d \biggr) \|v\|_{\Np(X)}
     < \eps + \frac{\eps}{d}.
\]
Let $w= {u}/{(1-\eps)} +\eta v$,
so that $ w\in\Np_0(E)$ and $w \ge 1$ on 
\[
     ((U \setm V) \cup V) \cap W = (U \cup V) \cap W,
\]
which is an open set containing $A$.
It follows that
\begin{align*}
   \inf_{\substack{G\text{ open} \\  
              A \subset G \subset E \\ }}
                            \cp(G,E)^{1/p} 
          &  \le \cp((U \cup V) \cap W,E)^{1/p} 
\le \|g_w\|_{L^p(X)}  \\
          & \le \frac{\|g_u\|_{L^p(X)}}{1-\eps} + \|g_{\eta v}\|_{L^p(X)}
  < \frac{(\cp(A,E)+\eps)^{1/p}}{1-\eps} +\eps + \frac{\eps}{d}.
\end{align*}
Letting $\eps\to0$ completes this part of the proof.

Let now $A\subset \interior E$ be arbitrary and $p>1$.
For $j=1,2,\ldots$, let
\[
A_j=\{x\in A: \dist(x,X\setm E)\ge1/j\}.
\]
Then the first part of the proof applies to $A_j$.
Let $\eps>0$ and for each $j=1,2,\ldots$, find an open set 
$G_j\subset E$
such that $A_j\subset G_j$ and
\[
\cp(G_j,E) \le \cp(A_j,E) + \eps_j,
\]
where $\eps_j=2^{-j}\eps$.
Let $\Gt_k = \bigcup_{j=1}^k G_j$ and 
$G=\bigcup_{k=1}^\infty \Gt_k = \bigcup_{j=1}^\infty G_j$.
Then $G$ is open and $A\subset G\subset E$.
Lemma~\ref{lem-Aj-Gj-eps} shows that for all $k=1,2,\ldots$,
\[
\cp(\Gt_k,E) \le \cp(A_k,E) + \sum_{j=1}^k \eps_j.
\]
Finally, by Theorem~\ref{thm-cp}\,\ref{cp-Choq-E-sum}
(it is here that we need that $p>1$) we get that
\[
\cp(G,E) = \lim_{k\to\infty} \cp(\Gt_k,E)
   \le \lim_{k\to\infty} \cp(A_k,E) + \sum_{j=1}^{\infty} \eps_j
   \le \cp(A,E) + \eps.
\]
Since $\eps$ was arbitrary, this finishes the proof.
\end{proof}

Let us now draw some consequences of the fact that
$\cp$ is an outer capacity.
We start by the following
characterization of our variational capacity, whose proof we leave to
the reader.

\begin{cor} \label{cor-altdef}
Assume that all functions in $\Np(X)$ are quasicontinuous\/
\textup{(}which in particular holds if 
$X$ is complete and supports a\/ $(1,p)$-Poincar\'e inequality,
and  $\mu$ is doubling\/\textup{)}.
Assume further that $A \subset \interior E$, if $p>1$,
or\/ $\dist(A, X \setm E)>0$, if $p=1$. 
Then 
\[
    \cp(A,E) = \inf_{u} \int_X g_u^p \, d\mu,
\]
where the infimum is taken over all functions $u \in \Np_0(E)$
such that  $u \ge 1$ in an open set containing $A$.
\end{cor}

If $E$ is measurable, then one may equivalently
integrate over $E$ instead.

Again, if we consider $A \subset E$ and replace ``open'' by ``relatively open''
the result is false in general but true sometimes,
see Examples~\ref{ex-arc} and~\ref{ex-dense} and 
Proposition~\ref{prop-cp=cpvar}.

Another consequence is the following equality between the capacity
of a set and its \p-fine closure. 
For open $E$ and $A\Subset E$ this was proved in J.~Bj\"orn~\cite{JB-pfine},
Corollary~4.5 (and is also included as Corollary~11.39 in
\cite{BBbook}).
(See e.g.\ \cite{BBbook} for the definition of the fine topology and other
concepts used in the proof below.)

\begin{cor}    \label{cor-capA-capAbar}
Assume that $p>1$, that 
$X$ is complete and supports a\/ $(1,p)$-Poincar\'e inequality,
and that  $\mu$ is doubling.
Let $A\subset\inter E$ and $\clAp$ be the \p-fine closure of
$A$, i.e.\ the smallest \p-finely closed set containing $A$.
Then $\clAp\subset\inter E$ and
\begin{equation}
\cp(\clAp,E) = \cp(A,E).
\label{cap-p-fine-closure}
\end{equation}
\end{cor}

\begin{proof}
If $\cp(A,E)=\infty$, then $\cp(\clAp,E)=\infty$.
We can therefore assume that $\cp(A,E)<\infty$.
One inequality is trivial.
To prove the other one, 
assume first that $A$ is open and let $u$ be 
a solution of the obstacle problem on $E$ with zero boundary data and 
obstacle $\chi_A$, i.e.\ $u\in\Np_0(E)$ satisfies $u\ge1$ q.e.\ on $A$ 
and minimizes the energy integral in the definition of $\cp(A,E)$.
Such a minimizer exists, and is unique up to sets of $\Cp$-capacity zero,
 by Theorem~4.2 in Bj\"orn--Bj\"orn~\cite{BBnonopen}.
Note that $\cp(A,E)=\int_E g_u^p \,d\mu $.
It is easily verified that $u$ is a superminimizer in $\inter E$,
and hence by
Theorem~5.1 in Kinnunen--Martio~\cite{KiMa02}
(or Theorem~8.22 in \cite{BBbook}), it can be redefined on
a set of zero $\Cp$-capacity so that it becomes
lower semicontinuously regularized, i.e.\
\[
           u(x):= \lim_{r \to 0} \essinf_{B(x,r)} u.
\] 
Proposition~7.6 in~\cite{KiMa02} 
(or Proposition~9.4 in \cite{BBbook})
then implies that $u$ is superharmonic 
in $\inter E$, and
Theorem~4.4 in J.~Bj\"orn~\cite{JB-pfine} 
(or Theorem~11.38 in \cite{BBbook})
shows that it is 
\p-finely continuous in $\inter E$. 
Thus, the set $\{x\in X:u(x)\ge1\}$ is \p-finely
closed and contains $A$, and thus  also $\clAp$.
It follows that $u$ is admissible in the definition of
$\cp(\clAp,E)$ and hence
\[
\cp(\clAp,E) \le \int_E g_{u}^p \,d\mu 
  = \cp(A,E),
\]
proving the corollary for open $A $. 
For general $A \subset \inter E$, Theorem~\ref{thm-outercap-cp} yields
\begin{align*}
\cp(A,E) &= \inf_{\substack{G \text{ open} \\  A\subset G \subset E}} \cp(G,E)
= \inf_{\substack{G \text{ open} \\  A\subset G \subset E}} \cp(\clGp,E)
\ge \cp(\clAp,E).\qedhere
\end{align*}
\end{proof}

\begin{thm} \label{thm-Ki}
Assume that all functions in $\Np(X)$ are quasicontinuous\/
\textup{(}which in particular holds if 
$X$ is complete and supports a\/ $(1,p)$-Poincar\'e inequality,
and  $\mu$ is doubling\/\textup{)}.
Let  $K_1 \supset K_2 \supset \cdots \supset
K:=\bigcap_{j=1}^\infty K_j $ be  compact subsets 
of\/ $\interior E$.
If $p=1$, we further require that\/ $\dist(K,X \setm E)>0$.
Then 
\begin{equation} \label{eq-Ki}
      \cp(K,E) 
          = \lim_{j \to \infty} \cp(K_j,E).
\end{equation}
\end{thm}

It is natural to ask what happens if we merely 
require that $K_1\supset K_2\supset\ldots$ are compact subsets of $E$.
In the situation described in Example~\ref{ex-dense} it follows
from those arguments that \eqref{eq-Ki} is true even if $K\not\subset\inter E$.
On the other hand, if we let
$K_j=[0,1/j]^2$ and $K=\{(0,0)\}$
in the situation described in Example~\ref{ex-arc},
we see that
$\cp(K_j,E)=\infty$ for $j=1,2,\ldots$, while
$\cp(K,E)=0$ for $1<p\le2$.

\begin{proof}
That $\cp(K,E) \le \lim_{j \to \infty} \cp(K_j,E)$
follows directly from monotonicity.

Conversely, let $G \supset K$ be open.
Then $G\cup \bigcup_{j=1}^\infty (X \setm K_j)$ is an open cover  of
the compact set $K_1$.
Thus, there is a finite subcover, i.e.\ an $N$ such that
\[
K_1\subset G \cup \bigcup_{j=1}^N (X \setm K_j)
= G \cup (X \setm K_N).
\]
As $K_N\subset K_1$, it follows that $K_N \subset G$.
So $\lim_{j \to \infty} \cp(K_j,E) \le \cp(G \cap \interior E, E)$.
By Theorem~\ref{thm-outercap-cp} we obtain the equality sought for.
\end{proof}

A set function satisfying the conditions in 
Theorem~\ref{thm-cp}\,\ref{cp-subset-sum}, \ref{cp-Choq-E-sum} 
and Theorem~\ref{thm-Ki} is
a \emph{Choquet capacity}.
More precisely, $\cp(\,\cdot\,,E)$ is a Choquet capacity for subsets
of $\interior E$.
An important consequence is the following result.

\begin{thm} \label{thm-Choq-cap}
\textup{(Choquet's capacitability theorem)}
Let $p>1$.
Assume that $X$ is locally compact  and 
that all functions in $\Np(X)$ are quasicontinuous\/
\textup{(}which in particular holds if 
$X$ is complete and supports a\/ $(1,p)$-Poincar\'e inequality,
and  $\mu$ is doubling\/\textup{)}.

Then, all Borel sets\/ \textup{(}and even all 
Suslin sets\/\textup{)}
$A \subset \interior E$ are \emph{capacitable}, i.e.
\begin{equation} \label{eq-choq}
     \cp(A,E)=\sup_{\substack{K \text{ compact} \\  K \subset A}} \cp(K,E)
      =\inf_{\substack{G \text{ open} \\  A\subset G \subset E}} \cp(G,E).
\end{equation}
\end{thm}

Suslin sets
are sometimes called analytic sets (although analytic sets in
complex analysis is an entirely different concept).
The interested reader should look elsewhere for more on 
Suslin sets, e.g.\ in Aikawa--Ess\'en~\cite{AE}, Part~2, Section~10.

\begin{proof}
To obtain the first equality in \eqref{eq-choq},
we apply Choquet's capacitability theorem in
its usual abstract formulation (for which we need that $\interior E$ is
locally compact), 
see e.g.\ Theorem~10.1.1 in \cite{AE}, Part~2.
The second equality follows from Theorem~\ref{thm-outercap-cp}.
\end{proof}

\section{Equivalence with the
definition in \texorpdfstring{$\R^n$}{}}
\label{sect-eq-Rn}

Our aim in this section is to show 
that our definition of the variational capacity based on 
Newtonian spaces is equivalent to the definitions
based on usual (and weighted) Sobolev spaces used in $\R^n$.
This probably belongs to folklore but does not seem to be written 
down anywhere. 
The proof in fact depends on a deep result due to Cheeger~\cite{Cheeg}
and on Choquet's capacitability theorem (Theorem~\ref{thm-Choq-cap})
together with Theorem~\ref{thm-outercap-cp}.
In unweighted $\R^n$, the use of Cheeger's theorem can be avoided by
more elementary methods, see e.g.\ 
Appendix~A.1 in Bj\"orn--Bj\"orn~\cite{BBbook}.

\begin{thm} \label{thm-equiv-HeKiMa}
Let\/ $\R^n$ be equipped with a \p-admissible weight $w$, $p>1$,
and let\/ $\Om \subset \R^n$ be a nonempty bounded open set.
Then our variational capacity\/ $\cp(\,\cdot\,,\Om)$ 
coincides with the variational capacity\/ $\cpmu(\,\cdot\,,\Om)$ in 
Heinonen--Kilpel\"ainen--Martio\/~\textup{\cite{HeKiMa}},
where $d\mu=w \, dx$.
\end{thm}

An arbitrary nonnegative  function $w$ on $\R^n$
 is a \emph{\p-admissible weight}, $p>1$, if
$d\mu:=w\,dx$ is doubling and $\R^n$ equipped
with $\mu$ supports a $(1,p)$-Poincar\'e inequality,
see Corollary~20.9 in \cite{HeKiMa}
(which is only in the second edition).
The \p-Poincar\'e inequality used there 
differs somewhat from our Definition~\ref{def-PI}, but by 
Proposition~A.17 in Bj\"orn--Bj\"orn~\cite{BBbook}
it is equivalent to it.

Let us recall how the capacity $\cpmu(\,\cdot\,,\Om)$
is defined in \cite{HeKiMa}, p.\ 27.
For compact $K \subset \Om$  one lets
\begin{equation} \label{eq-Cp-def-K}
     \cpmu(K,\Om)= \inf_{u} \int_\Om |\nabla u|^p \,d\mu,
\end{equation}
where the infimum is taken over all $u \in C_0^\infty(\Om)$
such that $u \ge 1$ on $K$.
The capacity is first extended to open $G \subset \Om$ by letting
\begin{equation} \label{eq-Cp-def-G}
     \cpmu(G,\Om) = \sup_{\substack{K \text{ compact} \\  K \subset G}} 
     \cpmu(K,\Om)
\end{equation}
and then to arbitrary $A \subset \Om$ by
\begin{equation} \label{eq-Cp-def-A}
     \cpmu(A,\Om)
      =\inf_{\substack{G \text{ open} \\  A\subset G \subset \Om}} 
    \cpmu(G,\Om).
\end{equation}

\begin{proof}[Proof of Theorem~\ref{thm-equiv-HeKiMa}.]
Let $X$ be $\R^n$ equipped with the measure $d\mu=w\,dx$.
By Propositions~A.12 and~A.13 in~\cite{BBbook}
(whose proofs depend on a deep result of Cheeger~\cite{Cheeg}), we have
$g_u=|\nabla u|$ a.e.\ for all $u\in \Np(X)$,
where $\nabla u$ is the weak Sobolev gradient of $u$
as defined in \cite{HeKiMa}.
(If $\R^n$ is unweighted, then $\nabla u$ is the distributional gradient.)

Let $\Om\subset X$ be a bounded open set and $K\subset\Om$ compact.
Theorem~6.19\,(x) in~\cite{BBbook} 
(or Theorem~1.1 in Kallunki--Shanmugalingam~\cite{KaSh}) shows that
\begin{equation}   \label{eq-cp-K}
    \cp(K,\Om)=\inf_u \|g_u\|^p_{L^p(X)},
\end{equation}
where the infimum is taken over all Lipschitz functions $u$ on $X$ such 
that $u\ge1$ on $K$ and $u=0$ in $X\setm\Om$.
Replacing each such $u$ by $(1-\eps)^{-1}(u-\eps)_\limplus$ and
letting $\eps\to0$ implies that the infimum can equivalently be taken
over all Lipschitz functions $u$ with compact support in $\Om$ and 
$u\ge1$ on $K$.
Since $C^\infty_0(\Om)$-functions are Lipschitz, we can directly conclude
that
\(
\cp(K,\Om) \le \cpmu(K,\Om).
\)

Conversely, it follows from the comments on pp.\ 27--28 in
\cite{HeKiMa} that
\begin{equation} \label{eq-cpmu-K}
    \cpmu(K,\Om)=\inf_u \int_\Om |\nabla u|^p\,d\mu,
\end{equation}
where the infimum is taken over all continuous $u \in \HP_0(\Om,\mu)$
such that $u \ge 1$ on $K$. 
Here $\HP_0(\Om,\mu)$ is the closure of $C^\infty_0(\Om)$ in the Sobolev norm
$\|u\|_{L^p(\Om,\mu)} + \|\nabla u\|_{L^p(\Om,\mu)}$.
As Lipschitz functions with compact support in $\Om$ belong 
to $\HP_0(\Om,\mu)$, by Lemma~1.25 in~\cite{HeKiMa},
we immediately get from~\eqref{eq-cp-K} and~\eqref{eq-cpmu-K} that 
\(
\cp(K,\Om) \ge \cpmu(K,\Om).
\)
Thus, the capacities coincide for compact sets.

For open and arbitrary subsets of $\Om$, the result now follows
from the definitions~\eqref{eq-Cp-def-G} and~\eqref{eq-Cp-def-A}
together with Choquet's capacitability theorem (Theorem~\ref{thm-Choq-cap})
and Theorem~\ref{thm-outercap-cp}.
\end{proof}

In Mal\'y--Ziemer~\cite{MaZi}, p.\ 63, the variational capacity on 
unweighted $\R^n$ is defined directly for arbitrary $A\subset\Om$ by taking
the infimum in the \p-energy integral over all $u$ in the Sobolev space
$\HP_0(\Om)$ defined above, such that 
$u \ge 1$ in a neighbourhood of~$A$.
A similar definition can be made for weighted $\R^n$ as well.
Using this definition, the equivalence with our capacity $\cp(A,\Om)$
can be proved without the use of Choquet's capacitability theorem.
All that is needed is the equality $g_u=|\nabla u|$ (provided essentially
by Cheeger's theorem) and the fact that
\[
        \HP_0 (\Om,\mu) = \{u : u=v \text{ a.e. for some } v \in \Np_0(\Om)\},
\]
i.e.\ that $\Np_0(\Om)$ consists exactly of the quasicontinuous
representatives of functions from $\HP_0 (\Om,\mu)$, see Proposition~A.13
in~\cite{BBbook} and Theorem~4.5 in~\cite{HeKiMa}.

\section{Other definitions and applications of capacity}
\label{sect-alt-def}

In $\R^n$, capacity is often defined without using Sobolev spaces, as e.g.\ in
\eqref{eq-Cp-def-K}--\eqref{eq-Cp-def-A}.
This is sometimes possible also on metric spaces, when $E=\Om$
is open. 

If $X$ is complete and supports a $(1,p)$-Poincar\'e inequality,
$\mu$ is doubling and $p>1$, then
Theorem~1.1 in Kallunki--Shanmugalingam~\cite{KaSh}
(or Theorem~6.19\,(x) in~\cite{BBbook}) shows that
for compact sets $K\subset\Om$, the capacity 
$\cp(K,\Om)$ can be defined using only Lipschitz functions with 
compact support in $\Om$, i.e.\ $u\in\Lip_c(\Om)$.
Theorem~6.1 in Cheeger~\cite{Cheeg} shows that under the same assumptions, 
$g_u=\Lip u=\lip u$ a.e., where
\begin{align*}
   \Lip u(x) &:= \limsup_{r\to0} \sup_{y\in B(x,r)} \frac{|u(y)-u(x)|}{r}
 \intertext{and} 
   \lip u(x) &:= \liminf_{r\to0} \sup_{y\in B(x,r)} \frac{|u(y)-u(x)|}{r}
\end{align*}
are the \emph{upper} and \emph{lower pointwise dilations} of $u$, respectively.
Thus, $g_u$ in the definition of $\cp(K,\Om)$ 
can be replaced by $\Lip u$ or $\lip u$
and $\cp(K,\Om)$ can be defined using only elementary properties of
Lipschitz functions, i.e.
\begin{equation} \label{cp-Lip-def}
\cp(K,\Om)
       = \inf_{\substack{u \ge 1 \text{ on } K \\ u \in \Lip_c(\Om)}}  
         \int_{\Om} (\Lip u)^p \, d\mu
       = \inf_{\substack{u \ge 1 \text{ on } K \\ u \in \Lip_c(\Om)}}  
         \int_{\Om} (\lip u)^p \, d\mu.
\end{equation}
Equivalently, $\Lip_c(\Om)$ can be replaced by 
$\Lip_0(E):=\{f \in \Lip(X) : f=0 \text{ on } X \setm \Om\}$.
The capacity can then be extended to open and arbitrary sets 
as in~\eqref{eq-Cp-def-G} and~\eqref{eq-Cp-def-A}.
By Theorems~\ref{thm-outercap-cp} and~\ref{thm-Choq-cap}, this is
equivalent to Definition~\ref{deff-varcap}, provided that
$p>1$, $X$ is complete and supports a $(1,p)$-Poincar\'e inequality,
and  $\mu$ is doubling.

It is natural to ask if \eqref{cp-Lip-def} may be extended to
nonopen sets, i.e.\ if $\Om$ can be replaced by an arbitrary
$E$ in \eqref{cp-Lip-def}.
(If $E$ is not measurable we take the integrals over $X$.)
If $K \not \subset \interior E$, then the equality can hold
only when $\cp(K,E)=\infty$, as there are no Lipschitz
functions satisfying the requirements in the infima.
The following example shows that the equality is not true (in general) even for 
$K \subset \interior E$.
Thus for general $E$ we are better off using Newtonian functions in 
the definition of $\cp(A,E)$.

\begin{example}
Let $E=\Om \setm D \subset X=\R^n$ (unweighted), $1< p \le n$,
where $D\subset\Om$ is a countable set whose closure has positive Lebesgue
measure.
Assume also that $\interior E=\Om\setm\itoverline {D}\ne\emptyset$
and let $K \subset \interior E$ be compact.
As in Example~\ref{ex-dense} we see that
$
         \cp(K,E)=\cp(K,\Om),
$
while every Lipschitz function with compact support in $E$ 
must vanish on $\itoverline{D}$ and hence
\begin{align*}
\inf_{\substack{u \ge 1 \text{ on } K \\ u \in \Lip_c(E)}}  
         \int_{\Om} (\Lip u)^p \, d\mu 
       &= \inf_{\substack{u \ge 1 \text{ on } K \\ u \in \Lip_c(\interior E)}}  
         \int_{\Om} (\Lip u)^p \, d\mu
       =        \cp(K,\interior E),
\end{align*}
and similarly for $u\in\Lip_0(E)$.
Since for most compact sets $K \subset\interior E$ we have
$\cp(K,\Om) <\cp(K,\interior E)$, this shows that 
\eqref{cp-Lip-def} cannot extend to the nonopen case.
\end{example}

Let us now return to the capacity $\cpvar$ from~\eqref{eq-def-cpvar} in
the introduction. 
By definition, it is an outer capacity, in the sense that
\begin{equation} \label{eq-cpvar}
\cpvar(A,E)=\inf_{\substack{G \text{ relatively open} \\  A\subset G \subset E}} 
    \cpvar(G,E).
\end{equation}
holds for every $A\subset E$.

It is fairly easy to establish \ref{cp-emptyset-sum}, 
\ref{cp-subset-sum} and \ref{cp-strong-subadd}--\ref{cp-F-bdyF} 
of Theorem~\ref{thm-cp} for $\cpvar$:
The parts \ref{cp-emptyset-sum} and \ref{cp-subset-sum} are trivial,
\ref{cp-strong-subadd} and \ref{cp-subadd-sum} follow from the corresponding
properties for $\cp$, while \ref{cp-Choq-E-sum} 
is proved in the same way as in Theorem~\ref{thm-cp} using relatively
open $G_i\supset A_i$ with $\cp(G_i,E)<\cpvar(A_i,E)+2^{-i}\eps$
and functions $u_i\in\Np_0(E)$ such that $\chi_{G_i}\le u_i\le1$ and
$\|g_{u_i}\|_{L^p(X)}<\cpvar(A_i,E)+2^{-i}\eps$,
and the proof of \ref{cp-F-bdyF} is similar using open $G \supset \bdy F$
and the technique in the proof of  Theorem~\ref{thm-cp}.
We omit the details here.

The strong subadditivity (Theorem~\ref{thm-cp}\,\ref{cp-strong-subadd})
for $\cpvar$ also implies that Lemma~\ref{lem-Aj-Gj-eps} holds for $\cpvar$.
Lemma~\ref{lem-Cp<=>cp} is however not true for $\cpvar$, see
Example~\ref{ex-arc}.
The following example shows that
Theorem~\ref{thm-cp}\,\ref{cp-subset-sum-2} for $\cpvar$ is not true either 
in general.
However, if $E_1$ is relatively open in $E_2$ then every $G\supset A$ which is
relatively open in $E_1$ is also relatively open in $E_2$ and hence 
Theorem~\ref{thm-cp}\,\ref{cp-subset-sum-2} holds for $\cpvar$, by the same
property for $\cp$.

\begin{example}   \label{ex-bow-tie}
Let $1<p \le 2$,
\begin{alignat*}{2}
    X &= \{(x,y)\in [-2,2]^2: xy\ge0\}, &\quad  A & = \{(0,0)\}, \\
    E_1 &= [0,1)^2, &\quad
    E_2 &= E_1\cup\{(x,y)\in X: x\le y\le0\}.
\end{alignat*}
Then $\cpvar(A,E_1)=0$ since $\Cp(A)=0$.
At the same time, every open $G\supset A$ must contain a segment from 
the boundary $\bdry E_2$ and since functions in $\Np([-2,0]^2)$ are absolutely
continuous on \p-almost every curve, we see that $\cp(G,E_2)=\infty$
and hence $\cpvar(A,E)=\infty$.

Note that in this example, all $u\in\Np(X)$ are quasicontinuous (by e.g.\ 
Example~5.6 and Theorem~5.29 in~\cite{BBbook}), 
but 
the zero \p-weak upper gradient property fails at the origin, 
cf.\ Proposition~\ref{prop-cp=cpvar} below.
\end{example}

If $K_1 \supset K_2 \supset \cdots \supset
K:=\bigcap_{j=1}^\infty K_j $ are  compact subsets of $E$,
then the inner regularity
\[
      \cpvar(K,E) 
          = \lim_{j \to \infty} \cpvar(K_j,E)
\]
can be shown in the same way as Theorem~\ref{thm-Ki}, where we use
\eqref{eq-cpvar} instead of Theorem~\ref{thm-outercap-cp} (and that
is also why we can allow for $K_j\subset E$ here rather than only
$K_j\subset\interior E$ as in Theorem~\ref{thm-Ki}).
Thus $\cpvar$ is a Choquet capacity if $p>1$, and we can establish
Choquet's capacitability theorem in the following form.
Note that to obtain these properties for $\cpvar$ there is no
need to assume that all functions in $\Np(X)$  are quasicontinuous,
since outer regularity of $\cpvar$ comes for free rather than from
Theorem~\ref{thm-outercap-cp}.

\begin{thm} \label{thm-Choquet-cpvar}
\textup{(Choquet's capacitability theorem for $\cpvar$)}
Let $p>1$.
Then, all Borel sets\/ \textup{(}and even all 
Suslin sets\/\textup{)}
$A \subset E$, for which there exists a locally compact set $F$
such that $A \subset F \subset E$, are \emph{capacitable}, i.e.
\begin{equation} \label{eq-choq-2}
     \cpvar(A,E)=\sup_{\substack{K \text{ compact} \\  K \subset A}} \cpvar(K,E)
      =\inf_{\substack{G \text{ relatively open} \\  A\subset G \subset E}} \cpvar(G,E).
\end{equation}
\end{thm}

Note that if $E$ is locally compact, in particular if $E$ is open or compact,
then \eqref{eq-choq-2} holds for all $A \subset E$ (provided that $p>1$).

\begin{proof}
Restrict  $\cpvar(\,\cdot\,,E)$ to subsets of $F$.
It is then clear that this restricted capacity is  a 
Choquet capacity on $F$.
We can now apply Choquet's capacitability theorem in
its usual abstract formulation (for which we need that $F$ is
locally compact), see e.g.\ Theorem~10.1.1 in Aikawa--Ess\'en~\cite{AE}, Part~2.
This gives the first equality in \eqref{eq-choq-2}, while
the second equality is just \eqref{eq-cpvar}.
\end{proof}

\begin{remark} \label{rmk-cp=cpvar}
It follows directly from the definition of $\cpvar$ that
  $\cpvar(A,E)=\cp(A,E)$  for relatively open subsets $A$
of $E$, but not for general subsets of $E$, see Examples~\ref{ex-arc} 
and~\ref{ex-bow-tie}.

If all functions in $\Np(X)$ are quasicontinuous, then 
$\cpvar(A,E)=\cp(A,E)$ if $A \subset \inter E$ and $p>1$, 
or $\dist(A, X\setm E)>0$ and $p=1$, by Theorem~\ref{thm-outercap-cp}
 and the fact that
$G\cap\interior E$ is open for every relatively open $G\subset E$,
\end{remark}

In fact, we have the following result which sheds some more light on
the equality $\cpvar=\cp$ and  the
question posed in~\eqref{eq-outercap-cp-rel}.
It depends on the zero \p-weak upper gradient property, which was
introduced in A.~Bj\"orn~\cite{ABcluster}, where it was also shown that 
it follows from the $(1,p)$-Poincar\'e inequality.

By definition, $X$ has the \emph{zero \p-weak
upper gradient property} if every measurable function $f$, 
which has zero as a \p-weak upper gradient in some ball $B(x,r)$, 
is  essentially constant in some (possibly smaller) ball  $B(x,\de)$, 
which can depend both on $f$ and $B(x,r)$.
(It is equivalent to require this for bounded measurable functions,
see Remark~5.8 in Bj\"orn--Bj\"orn~\cite{BBnonopen}.)

\begin{prop}   \label{prop-cp=cpvar}
Let $p>1$. 
Assume  that 
all functions in $\Np(X)$ are quasicontinuous and that $X$ 
has the zero \p-weak upper gradient property\/
\textup{(}both of which hold in particular if 
$X$ is complete and supports a\/ $(1,p)$-Poincar\'e inequality,
and  $\mu$ is doubling\/\textup{)}.
If $A\subset E$ then\/ $\cpvar(A,E)=\cp(A,E)$ or $\cpvar(A,E)=\infty$.
\end{prop}

For situations when $\cp(A,E)<\infty=\cpvar(A,E)$ see Examples~\ref{ex-arc}
and~\ref{ex-bow-tie}.
In both examples we have $\cp(A,E)=0$ but by adding 
an open set $V\subset E$ with $\cp(V,E)<\infty$ to $A$ we get
\[
0<\cp(A\cup V,E)<\infty=\cpvar(A\cup V,E).
\]

Note that by the proof below we see that
$\cp(A,E)=\cpvar(A,E)$ in case~2, while $\cpvar(A,E)=\infty$ in case~1.

\begin{proof}
It is clear that $\cpvar(A,E)\ge\cp(A,E)$ for all $A\subset E$.
To prove the converse inequality, assume that $\cp(A,E)<\infty$.
We shall distinguish two cases:

\emph{Case}~1. 
\emph{There exists $x\in A$ such that for all $r>0$ both $\Cp(B(x,r)\setm E)>0$
and $\Cp(B(x,r)\cap E)>0$.}
We shall show that in this case, $\cpvar(A,E)=\infty$.
Let $G\subset X$ be an arbitrary open set containing $A$
and find a ball $B=B(x,r)\subset G$. 
Assume that $u\in\Np_0(E)$ is such that $u=1$ in $G\cap E$.
Then $u\in\Np(B)$, $u=0$ in $B\setm E$ and $u=1$ in $B\cap E$.
In particular, $g_u=0$ a.e.\ in $B$.
The zero \p-weak upper gradient property implies that $u$ is 
essentially (and thus q.e.) constant in some smaller ball $B(x,\de)$.
This contradicts the choice of $x$ and $u$ 
and hence there are no $u\in\Np_0(E)$ admissible 
in the definition of $\cp(G\cap E,E)$, i.e.\  $\cp(G\cap E,E)=\infty$.
Since $G\supset A$ was arbitrary, we conclude that $\cpvar(A,E)=\infty$.

\emph{Case}~2. 
\emph{For every $x\in A$ there exists a ball $B_x\ni x$ such that 
$\Cp(B_x\setm E)=0$ or $\Cp(B_x\cap E)=0$.}
As $X$ is separable, the Lindel\"of property,
see Proposition~1.6 in Bj\"orn--Bj\"orn~\cite{BBbook},
 implies that $A$ can be covered 
by countably many of these balls, i.e.\ $A\subset\bigcup_{i=1}^\infty B_{x_i}$. 
Let $G'$ be the union of the balls $B_{x_i}$ for which $\Cp(B_{x_i}\setm E)=0$,
and $G''$ be the union of the remaining balls $B_{x_i}$ in the countable subcover.
Then $\Cp(G'\setm E)=0$ and $\Cp(G''\cap E)=0$.

As $A\cap G'\subset \interior(E\cup G')$, we have by 
Theorem~\ref{thm-cp}\,\ref{cp-subset-sum} and~\ref{cp-subset-sum-2},
Remark~\ref{rmk-cp=cpvar} (it is here we use that $p>1$)
and the definition of $\cpvar$ that
\begin{align}  \label{eq-cp-cpvar-A-G'}
\cp(A,E) &\ge \cp(A\cap G',E\cup G') = \cpvar(A\cap G',E\cup G') \nonumber\\
  &= \inf_{\substack{G \text{ open} \\  A\cap G'\subset G}} 
                     \cp(G\cap(E\cup G'),E\cup G').
\end{align}
Let $G\subset G'\cup G''$ be an open set in $X$ containing $A\cap G'$.
We shall show that
\begin{equation}  \label{eq-to show-EcupG'}
\cp(G\cap E,E) \le \cp(G\cap(E\cup G'),E\cup G').
\end{equation}
Together with \eqref{eq-cp-cpvar-A-G'} this then yields
\[
\cp(A,E) \ge \inf_{\substack{G \text{ open} \\  A\cap G'\subset G}} \cp(G\cap E,E)
     = \cpvar(A\cap G',E).
\]
Thus, for every $\eps>0$ there exists an open set $G\supset A\cap G'$ and 
$u\in\Np_0(E)$ such that $u\ge1$ in $G\cap E$ and 
\[
\int_X g_u^p\,d\mu \le \cp(A,E)+\eps.
\]
Since $\Cp(G''\cap E)=0$, we can modify $u$ on $G''\cap E$ to get
$u=1$ on the relatively open set $(G\cap E)\cup(G''\cap E)\supset A$.
Thus, $\cpvar(A,E)\le \cp(A,E)+\eps$ and letting $\eps\to0$ will prove the
proposition.

It remains to show~\eqref{eq-to show-EcupG'}.
Let $u\in\Np_0(E\cup G')$ be such that $u=1$ in $G\cap(E\cup G')$.
Since $\Cp(G'\setm E)=0$, we see that $u\in\Np_0(E)$ and is thus
admissible in the definition of $\cp(G\cap E,E)$.
Taking infimum over all such $u$ proves~\eqref{eq-to show-EcupG'} and
finishes the proof.
\end{proof}

The variational capacity $\cp(A,E)$ depends very much on the
underlying metric space $X$, even though we have refrained from making
this dependence explicit in the notation.
Let us however define $\cp(A,E;X):=\cp(A,E)$ and see how
changing $X$ can be of use.

If $A\subset E\subset X_1\subset X_2$ then $\Np(X_2)\subset\Np(X_1)$ and
we immediately obtain that $\cp(A,E;X_1)\le\cp(A,E;X_2)$.
The inequality can be strict and in particular it can happen that
$\cp(A,E;X_1)=0<\cp(A,E;X_2)$, even for open $E$.
Since the definition of $\Np(X)$ depends on curves in $X$, the capacity 
$\cp$ is influenced by the path-connectedness properties of the 
underlying space. 
In Bj\"orn--Bj\"orn--Shanmugalingam~\cite{BBS-Dir}, similar phenomena
for Sobolev capacities are used to obtain new resolutivity results
for the Dirichlet problem for \p-harmonic functions.
We refer the reader to the examples therein.

In the following example we briefly comment on some other properties of
$\cp$ with respect to different underlying spaces, as well as on
the influence of the underlying space on the minimizers 
in the definition of $\cp$.

\begin{example}   \label{ex-vary-X}
Let for instance $X$ be an open set $G\subset\R^n$, equipped with
the induced metric and measure,
where $\R^n$ may be unweighted or weighted using a \p-admissible weight.
Let further, for simplicity, $K \subset \Om \subset G$, 
where $K$ is compact and $\Om$ is open and bounded.
If $\bdy_{\R^n} \Om \subset G$, then
it is fairly easy to see that 
$\cp(K,\Om;G)=\cp(K,\Om;\R^n)$.
On the other hand, when $\bdy_{\R^n} \Om \setm G$ is substantial,
the situation becomes different, as we shall now see. 

Usually, when calculating the variational capacity one more or less
solves a Dirichlet problem with zero boundary values on $\bdy \Om$
and boundary values $1$ on~$K$.
When regarding $\cp(K,\Om;G)$ 
as a problem in $\R^n$, it can be seen that it corresponds
to a mixed boundary value problem of the following type:
zero boundary values on $\bdy_{G} \Om$, boundary values $1$ on $K$,
and zero Neumann boundary condition
on $\bdy_{\R^n} \Om \setm \bdy_{G} \Om$, provided that
$\Om$ is smooth enough as a subset of $\R^n$.
If it is less smooth, then the same is true in a generalized sense,
making it possible to study problems with 
zero Neumann boundary condition in very general situations.
See e.g.\  the discussion in 
Section~1.7 and Example~8.18 in Bj\"orn--Bj\"orn~\cite{BBbook}.
Since $X=G$ is not complete (unless $G=\R^n$) this
gives a further motivation for studying nonlinear potential theory on
noncomplete spaces.

In this situation one may also consider  
$X=\itoverline{G}$ as the underlying metric space, 
where the closure is taken with respect to $\R^n$. 
The above discussion is more or less  the same for $G$ and $\itoverline{G}$ 
(but more care has to be taken in the formulations
near the boundary $\bdy_{\R^n} G$).

An advantage of $\itoverline{G}$ is that it is complete. 
At the same time, both $G$ and $\itoverline{G}$ 
may fail to support a Poincar\'e inequality,
and the measure may fail to  be doubling on $G$ or $\itoverline{G}$.
We can still of course use the properties in Theorem~\ref{thm-cp},
since they hold in full generality.
Also Theorem~\ref{thm-Choquet-cpvar} holds on $G$ and $\itoverline {G}$.

But for $X=\itoverline{G}$,
Theorems~\ref{thm-outercap-cp}, \ref{thm-Ki}, \ref{thm-Choq-cap}
and Corollary~\ref{cor-altdef} are not available in general.
On the other hand, by Theorem~\ref{thm-quasicont} (applied with $\R^n$ and $G$ 
in place of $X$ and $\Om$), we have
all of Theorems~\ref{thm-outercap-cp}, \ref{thm-Ki}, \ref{thm-Choq-cap}
and Corollary~\ref{cor-altdef} available for $X=G$.
If $G$ moreover has the zero \p-weak upper gradient property, then also
Proposition~\ref{prop-cp=cpvar} is available for $X=G$.
\end{example}


\begin{thebibliography}{99}

\bibitem{ABBSprime} \artprep{Adamowicz, T.,
        Bj\"orn, A., Bj\"orn, J. \AND Shan\-mu\-ga\-lin\-gam, N.}
        {Prime ends for domains in metric spaces}
        {\emph{Preprint}, 2012, {\tt arXiv:1204.6444}}

\bibitem{AE} \book{\auth{Aikawa}{H} \AND \auth{Ess\'en}{M}}
	{Potential Theory -- Selected Topics}
	{Lecture Notes in Math. {\bf 1633}, Springer, Berlin--Heidelberg, 1996}

\bibitem{ABclass} \art{\auth{Bj\"orn}{A}}
         {A regularity classification of boundary points
           for \p-harmonic functions and quasiminimizers}
        {J. Math. Anal. Appl.} {338} {2008} {39--47}

 \bibitem{ABcluster} \art{\auth{Bj\"orn}{A}}
         {Cluster sets for Sobolev functions and quasiminimizers}
         {J. Anal. Math.} {112} {2010} {49--77}

\bibitem{BBbook} \book{Bj\"orn, A. \AND Bj\"orn, J.}
        {\it Nonlinear Potential Theory on Metric Spaces}
    {EMS Tracts in Mathematics {\bf 17},
        European Math. Soc., Zurich, 2011}

\bibitem{BBnonopen} \artprep{\auth{Bj\"orn}{A} \AND \auth{Bj\"orn}{J}}	
	{Obstacle and Dirichlet problems on arbitrary nonopen sets
          in metric spaces, and fine topology}
        {\emph{Preprint}, 2012, {\tt arXiv:1208.4913}}

\bibitem{BBMP} \art{\auth{Bj\"orn}{A}, \auth{Bj\"orn}{J},
	\auth{M\"ak\"al\"ainen}{T}
        \AND \auth{Parviainen}{M}}
        {Nonlinear balayage on metric spaces}
        {Nonlinear Anal.} {71} {2009} {2153--2171}


\bibitem{BBP} \art{\auth{Bj\"orn}{A}, \auth{Bj\"orn}{J}
	\AND \auth{Parviainen}{M}}
        {Lebesgue points and the fundamental convergence theorem
        for superharmonic functions on metric spaces}
        {Rev. Mat. Iberoam.} {26} {2010} {147--174}

\bibitem{BBS} \art{Bj\"orn, A., Bj\"orn, J. \AND Shanmugalingam, N.}
        {The Dirichlet problem for \p-harmonic functions on metric spaces}
        {J. Reine Angew. Math.} {556} {2003} {173--203}

\bibitem{BBS5} \art{Bj\"orn, A., Bj\"orn, J. \AND Shan\-mu\-ga\-lin\-gam, N.}
        {Quasicontinuity of Newton--Sobolev functions and density of Lipschitz
        functions on metric spaces}
        {Houston J. Math.} {34} {2008} {1197--1211}

\bibitem{BBS-Dir} \artprep{Bj\"orn, A., Bj\"orn, J. 
                         \AND Shan\-mu\-ga\-lin\-gam, N.}
        {The Dirichlet problem for \p-harmonic functions with respect to the
        Mazurkiewicz boundary, and new capacities}
        {\emph{In preparation}}

\bibitem{BMarola} \art{Bj\"orn, A. \AND Marola, N.}
	{Moser iteration for (quasi)minimizers on metric spaces}
	{Manuscripta Math.} {121} {2006} {339--366}

\bibitem{BjIll}  \art{Bj\"orn, J.}
        {Boundary continuity for quasiminimizers on metric spaces}
        {Illinois J. Math.} {46} {2002} {383--403}

\bibitem{JB-Matsue} \artin{\auth{Bj\"orn}{J}}
        {Wiener criterion for Cheeger \p-harmonic
        functions on metric spaces}
        {{\it Potential Theory in Matsue}, 
        Advanced Studies in Pure Mathematics {\bf 44}, pp. 103--115,
        Mathematical Society of Japan, Tokyo, 2006}

\bibitem{JB-pfine} \art{Bj\"orn, J.}
        {Fine continuity on metric spaces}
        {Manuscripta Math.}{125}{2008}{369--381}

\bibitem{JBCalcVar} \art{\auth{Bj\"orn}{J}}
        {Necessity of a Wiener type condition for boundary regularity
          of quasiminimizers and nonlinear elliptic equations}
        {Calc. Var. Partial Differential Equations}
	{35} {2009} {481--496}

\bibitem{BMS} \art{\auth{Bj\"orn}{J}, \auth{MacManus}{P} 
	\AND  \auth{Shanmugalingam}{N}}
        {Fat sets and pointwise boundary estimates for \p-harmonic functions
        in metric spaces}
        {J. Anal. Math.}{85}{2001}{339--369}

\bibitem{Cheeg} \art{Cheeger, J.}
        {Differentiability of Lipschitz functions on metric spaces}
        {Geom. Funct. Anal.} {9} {1999} {428--517}

\bibitem{Fa} \art{\auth{Farnana}{Z}}
        {Pointwise regularity for solutions of double obstacle problems 
        on metric spaces}
        {Math. Scand.} {109} {2011}{185--200}

\bibitem{Haj03} \artin{\auth{Haj\l asz}{P}}
	{Sobolev spaces on metric-measure spaces}
	{\emph{Heat Kernels and Analysis on Manifolds\textup{,} Graphs and
	Metric Spaces\/ \textup{(}Paris\textup{,} 2002\/\textup{)}}, 
	Contemp. Math. {\bf 338}, pp. 173--218,
	Amer. Math. Soc., Providence, RI, 2003}

\bibitem{HaKo} \book{\auth{Haj\l asz}{P}
	\AND \auth{Koskela}{P}}
	{Sobolev met Poincar\'e}
	{{Mem. Amer. Math. Soc.} {\bf 145}:688 (2000)}

\bibitem{HeKiMa} \book{\auth{Heinonen}{J}, 
	\auth{Kilpel\"ainen}{T}
	\AND \auth{Martio}{O}}
        {Nonlinear Potential Theory of Degenerate Elliptic Equations}
        {2nd ed., Dover, Mineola, NY, 2006}

\bibitem{HeKo98} \art{Heinonen, J. \AND Koskela, P.}
 	{Quasiconformal maps in metric spaces with controlled geometry}
	{Acta Math.} {181} {1998} {1--61}

\bibitem{KaSh} \art{\auth{Kallunki [Rogovin]}{S}
	\AND \auth{Shanmugalingam}{N}}
        {Modulus and continuous capacity}
        {Ann. Acad. Sci. Fenn. Math.}{26}{2001}{455--464}

\bibitem{KilMa94} \art{Kilpel\"ainen, T. \AND Mal\'y, J.}
        {The Wiener test and potential estimates for quasilinear elliptic 
        equations}
        {Acta Math.}{172}{1994}{137--161}

\bibitem{KiMa02} \art{Kinnunen, J. \AND Martio, O.}
         {Nonlinear potential theory on metric spaces}
         {Illinois Math. J.} {46} {2002} {857--883}

\bibitem{KiMa03} \art{Kinnunen, J. \AND Martio, O.}
        {Potential theory of quasiminimizers}
        {Ann. Acad. Sci. Fenn. Math.} {28} {2003} {459--490}

\bibitem{korte-private} \private{\auth{Korte}{R}}
        {2007}

\bibitem{korte08} \art{\auth{Korte}{R}}
        {A Caccioppoli estimate and fine continuity for superminimizers
         on metric spaces} 
        {Ann. Acad. Sci. Fenn. Math.} {33} {2008} {597--604}

 \bibitem{KoMaSh} \art{\auth{Korte}{R}, \auth{Marola}{N}
 	\AND \auth{Shanmugalingam}{N}}
 	{Quasiconformality, homeomorphisms between
 	metric measure spaces preserving quasiminimizers, and uniform
 	density property}
 	{Ark. Mat.} {50} {2012} {111--134}

\bibitem{KoMc} \art{Koskela, P. \AND MacManus, P.}
        {Quasiconformal mappings and Sobolev spaces}
        {Studia Math.}{131}{1998}{1--17}

\bibitem{LiMa} \art{\auth{Lindqvist}{P} \AND
        \auth{Martio}{O}}
        {Two theorems of N. Wiener for solutions of quasilinear
          elliptic equations}
        {Acta Math.}{155}{1985}{153--171}

\bibitem{luzin} \art{\auth{Luzin}{N. N}}
	{Sur les propri\'et\'es des fonctions mesurables}
	{C. R. Acad. Sci. Paris} {154} {1912} {1688--1690}

\bibitem{MaZi} \book{Mal\'y, J. \AND Ziemer, W. P.}
        {Fine Regularity of Solutions of Elliptic Partial
          Differential Equations}
        {Math. Surveys and Monographs {\bf 51}, Amer. Math. Soc.,
           Providence, RI, 1997}

\bibitem{martioRep09} \art{\auth{Martio}{O}}
	  {Capacity and potential estimates for quasiminimizers}
	  {Complex Anal. Oper. Theory}{5}{2011}{683--699}

\bibitem{Maz70}
     \artnopt{\auth{Maz{\cprime}ya}{V. G}}
        {On the continuity at a boundary point of solutions of quasi-linear
        elliptic equations}
        {Vestnik Leningrad. Univ. Mat. Mekh. Astronom.}
        {25{\rm:13}} {1970} {42--55}  \rm (Russian).
        English transl.: {\it Vestnik Leningrad Univ. Math.}
        {\bf 3} (1976), 225--242.

\bibitem{Sh-rev} \art{Shanmugalingam, N.}
        {Newtonian spaces\textup{:} An extension of Sobolev spaces
        to metric measure spaces}
        {Rev. Mat. Iberoam.}{16}{2000}{243--279}

\bibitem{Sh-harm} \art{Shanmugalingam, N.} 
        {Harmonic functions on metric spaces}
        {Illinois J. Math.}{45}{2001}{1021--1050}

\bibitem{vitali} \art{\auth{Vitali}{G}}
	{Una propriet\`a delle funzioni misurabili}
	{R. Ist. Lombardo Sci. Lett. Rend.} {38} {1905} {599--603}

\end{thebibliography}
\end{document}